\documentclass[11pt,a4paper]{article}
\usepackage{latexsym}
\usepackage{amsthm,amsmath,amssymb,amstext}

\usepackage{graphicx} 
\usepackage{psfrag,color}

\def\XXint#1#2#3{{\setbox0=\hbox{$#1{#2#3}{\int}$}
     \vcenter{\hbox{$#2#3$}}\kern-.5\wd0}}

\renewcommand{\d}{\textup{d}}
\newcommand{\R}{\mathbb{R}}
\newcommand{\pr}[2]{\textup{P}_{#1}\left(#2\right)}

\newcommand{\set}[2]{\left\{#1\,\left|\;#2\right.\right\}}

\newcommand{\mint}[3]{\int\limits_{#1}#2\,\textup{d}#3}

\newcommand{\D}[1]{{\mathcal D\big((0,T),H^1(\Gamma(#1))\big)}}
\newcommand{\W}[1]{{W_{#1}(0,T)}}
\newcommand{\Ltwo}[1]{{L^2\left(\Gamma(#1)\right)}}
\newcommand{\Hone}[1]{{H^1(\Gamma(#1))}}
\newcommand{\Hmone}[1]{{H^{-1}(\Gamma(#1))}}

\newtheorem{Theorem}{Theorem}[section]

\newtheorem{Lemma}[Theorem]{Lemma}

\newtheorem{LemmaAndDef}[Theorem]{Lemma and Definition}

\newtheorem{Corollary}[Theorem]{Corollary}

\theoremstyle{definition}

\newtheorem{Remark}[Theorem]{Remark}

\newtheorem{Example}[Theorem]{Example}

\newtheorem{Assumption}[Theorem]{Assumption}

\numberwithin{equation}{section}

\def\mathref#1{\ifmmode\mathrm{(\ref{#1})}\else(\ref{#1})\fi}

\def\rhx2{\sqrt{1+ r_{h,x}^2}}

\def\L2{{L^2(\Gamma(t))}}

\def\Ld{{L^2(\Gamma^h(t))}}

\def\C1{{C^1(\Gamma)}}

\def\Ltwoh{{L^2(\Gamma^h)}}
\def\C{C_\textup{int}}
\def\div{\textup{ div}\:}

\def\C{c_\textup{int}}

\def\ldown{{(\cdot)_l}}
\def\lup{{(\cdot)^l}}

\title{Control-constrained parabolic optimal control problems on evolving surfaces -- theory and variational discretization}
\author{
Morten Vierling\footnote{Schwerpunkt Optimierung und Approximation,
Universit\"at Hamburg, Bundesstra{\ss}e 55, 20146 Hamburg, Germany.}
}

\date{\today}

\begin{document}

\maketitle

\begin{center}
 {\bf \LARGE  }
 \end{center}
 
 {\small {\bf Abstract:} 
 We consider control-constrained linear-quadratic optimal control problems on evolving hypersurfaces in $\mathbb R^{n+1}$. In order to formulate well-posed problems, we prove existence and uniqueness of weak solutions for the state equation, in the sense of vector-valued distributions. We then carry out and prove convergence of the variational discretization of a distributed optimal control problem. In the process, we investigate the convergence of a fully discrete approximation of the state equation, and obtain optimal orders of convergence under weak regularity assumptions. We conclude with a numerical example.
 }
 \\[2mm]

\pagenumbering{arabic}
\section{Introduction}\label{S:Intro}

We investigate parabolic optimal control problems on evolving material hypersurfaces in $\mathbb R^{n+1}$. 
 Following \cite{DziukElliott2007}, we consider a parabolic state equation in its weak form
\begin{equation}\label{E:Transport}
\frac{\textup{d}}{\textup{d}t}\mint{\Gamma(t)}{y\,\varphi}{\Gamma(t)}+\mint{\Gamma(t)}{\nabla_\Gamma y\cdot\nabla_\Gamma\varphi}{\Gamma(t)}=\mint{\Gamma(t)}{y\,\dot\varphi}{\Gamma(t)}+\mint{\Gamma(t)}{f\,\varphi}{\Gamma(t)}\,,
\end{equation}
where $\Gamma = \big\{\Gamma(t)\big\}^{t\in[0,T]}$ is a family of $C^2$-smooth, compact $n$-dimensional surfaces in $\R^{n+1}$, evolving smoothly in time with velocity $V$. Further assume $f$ sufficiently smooth and let $\dot \varphi=\partial_t\varphi+V\nabla \varphi $ denote the material derivative of a smooth test function $\varphi$.

We start by defining unique weak solutions for the state equation. The idea is to pull back the problem onto a fixed domain, introducing distributional material derivatives in the sense of \cite{LionsMagenes1968} and a $W(0,T)$-like solution space. As a consequence, a large part of the theory developed around $W(0,T)$ for fixed domains applies, compare for example  \cite{LionsMagenes1968} and \cite{Lions1971} .

An alternative approach to prove existence of weak solutions along the lines of \cite{Ladyzhenskaya1968} is taken in \cite{Schumacher2010}, that entirely avoids the notion of vector-valued distributions.

Recent works also deal with the discretization of \eqref{E:Transport}, both in space, compare \cite{DziukElliott2010_PP},  and time, see \cite{DziukLubichMansour2010} and \cite{DziukElliott2011_PP}. 

In \cite{DziukElliott2010_PP} order-optimal  error bounds of type $\sup_{t\in[0,T]}\|\cdot\|_{L^2(\Gamma(t))}$ are derived for the discretization of the state equation, assuming a slightly higher regularity of the state than is used in section \ref{S:FEDisc} and \ref{S:DGDisc}, where we derive $\left(\int_0^T\|\cdot\|^2_{L^2(\Gamma(t))}\d t\right)^\frac{1}{2}$-like bounds. A class of Runge-Kutta methods to tackle the space-discretized problem is investigated in \cite{DziukLubichMansour2010}, assuming among other things that one can evaluate $f$ in a point-wise fashion, i.e. that $f(t)\in L^2(\Gamma(t))$ is well defined. For a fully discrete approach and the according error bounds see  \cite{DziukElliott2011_PP}. There a backwards Euler method is considered for  time discretization whose implementation resembles our discontinuous Galerkin approach in Section \ref{S:DGDisc}. Yet while the approach in  \cite{DziukElliott2011_PP} ultimately leads to $\sup_{t\in[0,T]}\|\cdot\|_{L^2(\Gamma(t))}$-convergence, we allow for non-smooth controls and thus cannot expect to obtain such strong convergence estimates.

Basic facts on  control constrained  parabolic optimal control problems and their discretization can be found for example in \cite{Troeltzsch2005} and \cite{MeidnerVexler2008_II}, respectively.

The paper is structured as follows. We begin with a very short introduction into the setting in Section \ref{S:Setting}.
In order to formulate well posed optimal control problems we first proof the existence of an appropriate weak solution in Section \ref{S:WeakSolutions}, complementing the existence results from \cite{DziukElliott2007}. We then use the the results from Section \ref{S:WeakSolutions} in order to formulate control constrained optimal control problems in section \ref{S:OCP}. Afterwards, we examine the space- and time-discretization of the state equation in Sections \ref{S:FEDisc} and \ref{S:DGDisc}, before returning to the optimal control problems in Section \ref{S:VarDisc}. There we apply variational discretization in the sense of \cite{Hinze2005} to achieve fully implementable optimization algorithms. We end the paper by giving a numerical example in Section \ref{S:Example}.

\section{Setting}\label{S:Setting}
Before we can properly formulate \eqref{E:Transport}, let us introduce some basic tools  and clarify what  our assumptions are regarding the family $\big\{\Gamma(t)\big\}_{t\in[0,T]}$.
 \begin{Assumption}\label{A:Phi}
 The hypersurface $\Gamma_0=\Gamma(0)\subset\R^{n+1}$ is $C^2$-smooth and compact (i.e. without boundary). $\Gamma$ evolves along a $C^2$-smooth velocity field $V:\R^{n+1}\times[0,T]\rightarrow \R^{n+1}$ with flow $\bar\Phi:\R^{n+1}\times[0,T]^2\rightarrow \R^{n+1}$, such that its restriction 
 $\Phi^s_t(\,\cdot\,):\Gamma(s)\rightarrow\Gamma(t)$ is a diffeomorphism for every $s,t\in [0,T]$.
\end{Assumption}

The assumption gives rise to a second representation of $\Gamma(t)$ and in particular implies $\Gamma(t)$ to be orientable with a smooth unit normal field $\nu(\cdot,t)$. As a consequence,  the evolution of  $\Gamma$ can be described as the level set of the signed distance function $d$ such that 
$$
\Gamma(t)=\set{x\in\R^{n+1}}{d(x,t)=0}\,,
$$
as well as $|d(x,t)| = \textup{dist}(x,\Gamma(t))$ and $\nabla d(x,t)=\nu(x,t)$ for $x\in\Gamma(t)$. Further, we have $d(\cdot,t)\in C^2(\mathcal N_r(t))$ for some tubular neighborhood  $\mathcal N_r(t)=\set{x\in\R^{n+1}}{|d(x,t)|\le r}$ of $\Gamma(t)$. Due to the uniform boundedness of the curvature of $\Gamma(t)$ the radius $r>0$ does not depend on $t\in[0,T]$. The domain of $d$ is $\mathcal N_T=\bigcup_{t\in[0,T]}\mathcal N_r(t)\times\{t\}$ which is a neighborhood of $\bigcup_{t\in[0,T]}\Gamma(t)\times\{t\}$ in $\R^{n+2}$. 

Using $d$ we can define the projection
\begin{equation}\label{E:Projection}a_t:\mathcal N_r(t)\rightarrow\Gamma(t)\,,\quad a_t(x) = x-d(x,t)\nabla d(x,t)\,,\end{equation} 
which allows us to extend any function $\phi:\Gamma(t)\rightarrow\R$ to $\mathcal N_r(t)$ by $\bar \phi(x) = \phi(a_t(x))$. Hence we can represent the surface gradient in global exterior coordinates $\nabla_{\Gamma(t)} \phi=(I-\nu(\cdot,t)\nu(\cdot,t)^T)\nabla \bar\phi$ as the euclidean projection of the gradient of $\bar\phi$ onto the tangential space of $\Gamma(t)$. In the following we will write $\nabla_{\Gamma}$ instead of $\nabla_{\Gamma(t)}$, wherever it is clear which surface $\Gamma(t)$ the gradient relates to.


\newcommand{\sgrad}[1]{{\nabla_{\Gamma}}}
\newcommand{\sgradh}[1]{{\nabla_{\Gamma^h}}}
We are going to exploit existing results on vector-valued distributions, which we  recall here for completeness. 
 In order to define weak  derivatives consider  $\mathcal D((0,T))$, the space of real valued $C^\infty$-smooth functions with compact support in $(0,T)$. Fix $s\in[0,T]$. Each $y\in  L^2((0,T),H^{1}(\Gamma(s)))$ defines a vector-valued distribution $\mathcal T_y:\mathcal D((0,T))\rightarrow  H^1(\Gamma(s))$ through the $H^1(\Gamma(s))$-valued integral $\mint{[0,T]}{ y(t)\varphi(t)}{t}$. 
 
 Its distributional derivative is said to lie in $ L^2((0,T),H^{-1}(\Gamma(s)))$ iff it can be represented by $w\in  L^2((0,T),H^{-1}(\Gamma(s)))$ in the following sense
\begin{equation}\label{E:DistrDer}
\mathcal T_y'(\varphi)=\mint{[0,T]}{ y(t)\varphi'(t)}{t}=- \mint{[0,T]}{w(t)\varphi(t)}{t}\in H^1(\Gamma(s))\,,\quad\forall \varphi\in \mathcal D((0,T))\,,
\end{equation}
 and we write $y'=w$. Note that by $H^{-1}$ we denote the representation of the dual $(H^1)^*$ which arises from $L^2\supset H^1$ by completion.

\begin{Lemma}\label{L:WZeroT}
For $s\in[0,T]$, the space $$\W{s}=\set{v\in L^2((0,T),H^1(\Gamma(s)))}{v'\in L^2((0,T),H^{-1}(\Gamma(s)))}$$ with scalar product $\int_0^T \langle\cdot,\cdot\rangle_{H^1(\Gamma(s))}+\langle(\cdot)',(\cdot)'\rangle_{H^{-1}(\Gamma(s))}\d t$ is a Hilbert space.
\begin{enumerate}
\item $\W{s}$ is compactly embedded into $C([0,T],L^2(\Gamma(s)))$, the space of continuous $L^2$-valued functions. 
\item Denote by $\mathcal D([0,T],H^1(\Gamma(s)))$ the space of $C^\infty$-smooth $H^1(\Gamma(s))$-valued test functions on $[0,T]$. The inclusion $\mathcal D([0,T],H^1(\Gamma(s)))\subset \W{s}$ is dense.
\item For two functions $v,w\in\W{s}$ the product $\langle v(t),w(t)\rangle_{L^2(\Gamma(s))}$ is absolutely continuous with respect to $t\in [0,T]$ and
\[
\frac{\d}{\d t}\mint{\Gamma(s)}{v(t)w(t)}{\Gamma(s)}=\langle v',w\rangle_{H^{-1}(\Gamma(s)),H^1(\Gamma(s))}+\langle v,w'\rangle_{H^{1}(\Gamma(s)),H^{-1}(\Gamma(s))}\,,
\]
a.e. in $(0,T)$, and as a consequence there holds the formula of partial integration
\[
\mint{[r,t]}{\langle v',w\rangle_{H^{-1},H^1}}{\tau} = \langle v(t),w(t)\rangle_{L^2(\Gamma(s))}- \langle v(r),w(r)\rangle_{L^2(\Gamma(s))} -\mint{[r,t]}{\langle v,w'\rangle_{H^{1},H^{-1}}}{\tau}\,.
\]
\end{enumerate}
\end{Lemma}
For a proof of the lemma, see \cite[Ch. I,Thrms. 3.1 and 2.1]{LionsMagenes1968}. In fact one can use the formula of partial integration to prove the embedding into $C((0,T),L^2(\Gamma(s)))$, see \cite[Ch. 5,Thrm 3]{Evans1998}. For further references see \cite[Thm. 3.10]{Troeltzsch2005}.

Our approach to weak material derivatives relies on the following equivalent formulation of condition \eqref{E:DistrDer}
 \begin{equation}\label{WeakDer_equiv}\begin{aligned}
 \forall \varphi\in \mathcal D((0,&T),H^1(\Gamma(s))):\\
&\mint{[0,T]}{ \langle y(t),\varphi'(t)\rangle_{L^2(\Gamma(s))}+\langle w(t),\varphi(t)\rangle_{H^{-1}(\Gamma(s)),H^{1}(\Gamma(s))}}{t}=0\,,
\end{aligned} \end{equation}
which defines the weak derivative $y'=w$ of a function $y\in L^2((0,T),H^1(\Gamma(s)))$ via its $L^2((0,T),L^2(\Gamma(s)))$-scalar product with elements of $\mathcal D((0,T),H^1(\Gamma(s)))$.

The equality \eqref{WeakDer_equiv} follows from \eqref{E:DistrDer} by Lemma \ref{L:WZeroT}[\emph{3.}]. On the other hand \eqref{E:DistrDer} is a consequence of \eqref{WeakDer_equiv}. To see this, test \eqref{WeakDer_equiv} with $\psi v\in\mathcal D((0,T),H^1(\Gamma(s)))$, where $\psi\in\mathcal D((0,T))$ and $v\in H^1(\Gamma(s))$.

\section{Weak solutions}\label{S:WeakSolutions}

The scope of this section is to formulate appropriate function spaces and a related weak material derivative, in order to prove the existence of unique weak solutions of  \eqref{E:Transport} for quite weak right-hand sides $f$. 

We start by  defining the strong material derivative for smooth functions $f\in C^1(\mathbb R^{n+1}\times[0,T])$, namely the derivative
\begin{equation}\label{E:MatDerStrong}\dot f(x,t)=\frac{\d}{\d s}\Big|_{s=t}f(\Phi^t_s(x),s)=\nabla f(x,t) V(x,t)+\partial_t f(x,t)\,,\end{equation}
along trajectories of the velocity field $V$.
The material derivative has the following properties.
\begin{Lemma}\label{L:MatDerSmooth}
Let $f$ be sufficiently smooth. Then
\[
\frac{\d}{\d t}\mint{\Gamma(t)}{f}{\Gamma(t)}=\mint{\Gamma(t)}{\dot f+\div_{\Gamma}V}{\Gamma(t)},
\]
and
\[
\frac{\d}{\d t}\mint{\Gamma(t)}{\|\sgrad{t}f\|^2}{\Gamma(t)}=\mint{\Gamma(t)}{2\sgrad{t}f\cdot \sgrad{t}\dot f-2\sgrad{t}f(D_{\Gamma}V)\sgrad{t}f+\|\sgrad{t}f\|^2\div_{\Gamma}V}{\Gamma(t)}\,,
\]
with $\div_{\Gamma(t)}V = \sum_{i=1}^{n+1}\nabla_{\Gamma(t)}^iV^i$ and $(D_{\Gamma(t)}V)_{ij}=\nabla_{\Gamma(t)}^jV^i$.
\end{Lemma}
A proof and details can be found in the Appendix of \cite{DziukElliott2007}.
\begin{LemmaAndDef}\label{L:basics}
Let 
$
J^s_t(\,\cdot\,)= \det \textup{D}_{\Gamma(s)}\Phi^s_t(\,\cdot\,)
$
denote the Jacobian determinant of the matrix representation of $\textup{D}_{\Gamma(s)}\Phi^s_t(\,\cdot\,)$ with respect to orthogonal bases of the respective tangent space. By Assumption \ref{A:Phi} $J^s_t\in C^1([0,T]\times\Gamma(s))$ and there exists $C_J>0$, such that for all $s,t\in[0,T]$
$$
\frac{1}{C_J}\le \min_{\gamma\in\Gamma(s)}J^s_t(\gamma)\le \max_{\gamma\in\Gamma(s)}J^s_t(\gamma)\le C_J\,.
$$
Given Assumption \ref{A:Phi}, consider the family $\big\{L^2(\Gamma(t))\big\}_{t\in[0,T]}$. Then for $v\in L^2(\Gamma(t))$ we introduce the pull-back
$$
\phi^s_{t} v=v\big(\Phi^s_{t}(\,\cdot\,)\big)\in L^2(\Gamma(s))\,,
$$
which is a linear homeomorphism from $L^2(\Gamma(t))$ into $L^2(\Gamma(s))$ for any $s,t\in[0,T]$. Moreover $\phi^s_{t}$ is a linear homeomorphism from $H^1(\Gamma(t))$ into $H^1(\Gamma(s))$
. Thus finally the adjoint operator, ${\phi^s_{t}}^* :H^{-1}(\Gamma(s))\rightarrow H^{-1}(\Gamma(t))$ is also a linear homeomorphism. 
There exist constants $C_{L^2(\Gamma)},C_{H^1(\Gamma)}$ independent of $s,t$, such that for all $v\in L^2(\Gamma(t))$, or $v\in H^1(\Gamma(t))$ respectively, and for all $s,t\in[0,T]$
$$
\|\phi^s_{t}v\|_{H^1(\Gamma(s))}\le C_{H^1(\Gamma)}\|v\|_{H^1(\Gamma(t))}\,,\quad\|\phi^s_{t}v\|_{L^2(\Gamma(s))}\le C_{L^2(\Gamma)}\|v\|_{L^2(\Gamma(t))}\,,
$$
and thus finally $\|{\phi^s_{t}}^*\|_{\mathcal L(H^{-1}(\Gamma(s)),H^{-1}(\Gamma(t)))}\le C_{H^{1}(\Gamma)}$.

Furthermore there holds $\partial_t J_t^s = \phi_t^s( \div_{\Gamma(t)} V) J_t^s$.
\end{LemmaAndDef}

\begin{proof}
For $s,t\in[0,T]$ we have
$$
\mint{\Gamma(t)}{v^2}{\Gamma(t)}=\mint{\Gamma(s)}{(\phi^s_{t} v)^2 J^s_t}{\Gamma(s)}
$$
and thus $\|\phi^s_{t} v\|_{L^2(\Gamma(s))}\le C_{L^2(\Gamma)}\|v\|_{L^2(\Gamma(t))}$, with $C_{L^2(\Gamma)}= C_J^{\frac{1}{2}}$.

For $H^1$ equivalence consider $v\in H^1(\Gamma(t))$ and choose $\varphi\in C^1(\Gamma(s))$. Now
$$
\mint{\Gamma(s)}{(\phi^s_{t} v)\sgrad{s}\varphi}{\Gamma(s)}=\mint{\Gamma(t)}{v \,(\text{D}\bar\Phi^s_t)^T\sgrad{t}(\phi^t_s\varphi)J^t_s}{\Gamma(t)}
$$
and because $v\in H^1(\Gamma(t))$ we can integrate by parts on $\Gamma(t)$ to obtain with $\nu_s=\nu(\cdot,s)$
$$
\mint{\Gamma(s)}{(\phi^s_{t} v)\sgrad{s}\varphi}{\Gamma(s)}=-\mint{\Gamma(t)}{w \,(\phi^t_s\varphi)J^t_s}{\Gamma(t)}=-\mint{\Gamma(s)}{(\phi^s_{t} w-H_{s}\nu_{s} \phi^s_tv)\varphi+(\phi^s_tv)\varphi H_{s}\nu_{s} }{\Gamma(s)}\,.
$$
Note that $w\in L^2(\Gamma(t))^{n+1}$ and that $\|w\|_{L^2(\Gamma(t))^{n+1}}\le C \|v\|_{H^1(\Gamma(t))}$, where $C$  depends only on the mean curvature $H_t$ of $\Gamma(t)$ and the second space derivatives of $\bar\Phi$ which are bounded independently of $s,t\in [0,T]$. Now $\nabla_{\Gamma(s)}(\phi^s_{t} v)=\phi^s_{t}w-H_{s}\nu_{s} \phi^s_tv\in L^2(\Gamma(s))$, because as stated above $\|\phi^s_{t} w\|_{L^2(\Gamma(s))^{n+1}}\le C_{L^2(\Gamma)}C\|v\|_{H^1(\Gamma(t))}$, and $\|\phi^s_{t} v\|_{L^2(\Gamma(s))}\le C_{L^2(\Gamma)}\|v\|_{L^2(\Gamma(t))}$. Thus, for some $C_{H^1(\Gamma)}>0$ depending only on a global  bound on $|H_t|$, $\|\partial_{i}\bar \Phi^s_t\|$ and $\|\partial_{ij}\bar \Phi^s_t\|$, $s,t\in[0,T]$, $1\le i,j\le n+1$, there holds
$$
\|\phi^s_{t} v\|_{H^1(\Gamma(s))}\le C_{H^1(\Gamma)} \|v\|_{H^1(\Gamma(t))}\,.
$$

Now $\|\,\cdot\,\|_{H^1(\Gamma(t))}$ and $\|\phi^s_{t} (\,\cdot\,)\|_{H^1(\Gamma(s))}$ are two equivalent norms on $H^1(\Gamma(t))$. Hence also their dual norms are equivalent. The norm of  $f\in \left( H^1(\Gamma(s))\right)'$ can now be expressed by
\begin{equation}\label{E:DualNormEquiv}
\sup_{w\in H^1(\Gamma(s))}\frac{\langle f ,w\rangle_{(H^1(\Gamma(s)))',H^1(\Gamma(s))}}{\|w \|_{H^1(\Gamma(s))}}=\sup_{v\in H^1(\Gamma(t))}\frac{\langle {\phi^s_{t}}^*f ,v\rangle_{(H^1(\Gamma(t)))',H^1(\Gamma(t))}}{\|\phi^s_{t} v \|_{H^1(\Gamma(s))}}\,,
\end{equation}
and the bound on the norm of ${\phi_t^s}^*$ follows from the equivalence of said $H^1$-norms.

The last assertion is a by-product of the proof of Lemma \ref{L:MatDerSmooth}, compare \cite{DziukElliott2007}.
\end{proof}
We need to state one more Lemma concerning continuous time-dependence of the previously defined norms.
\begin{Lemma}\label{L:NormContinuity}
Let $s\in[0,T]$. For $v_1\in H^1(\Gamma(s))$, $v_2\in L^2(\Gamma(s))$, $v_3\in H^{-1}(\Gamma(s))$ the following expressions are continuous with respect to $t\in[0,T]$
\[
\|\phi^t_{s}v_1\|_{H^1(\Gamma(t))}\,,\quad\|\phi^t_{s}v_2\|_{L^2(\Gamma(t))}\,,\quad\|{\phi^s_{t}}^*v_3\|_{H^{-1}(\Gamma(t))}\,.
\]
\end{Lemma}
\begin{proof}
By the change of variables formula we have 
\begin{equation}\label{H1Cont}
\|\phi^t_{s}v_1\|_{H^1(\Gamma(t))}^2=\mint{\Gamma(s)}{\left(\nabla_\Gamma v_1(D_{\Gamma(s)}\bar\Phi^s_t)^{-1}(D_{\Gamma(s)}\bar\Phi^s_t)^{-T} \nabla_\Gamma v_1+v_1^2\right)J^s_t}{\Gamma(s)}\,,
\end{equation}
which is a continuous function due to the regularity of $\Phi$ stated in Assumption \ref{A:Phi}. Similarly we conclude the continuity of the $L^2$-norm.

Moreover, since $(D_{\Gamma(s)}\bar\Phi^s_s)^{-1}(D_{\Gamma(s)}\bar\Phi^s_s)^{-T}=\textup{id}_{T\Gamma(s)}$, $J^s_s=1$, and $\Phi^s_{(\cdot)}(\cdot)\in C^2(\Gamma(s)\times[0,T])$ Equation  \eqref{H1Cont} infers
$$|\|\phi^t_{s}v\|_{H^1(\Gamma(t))}^2-\|v\|_{H^1(\Gamma(s))}^2|\le C|t-s|\|v\|^2_{H^1(\Gamma(s))}\,,$$
for all $v\in H^1(\Gamma(s))$. Regarding \eqref{E:DualNormEquiv} this allows us to estimate
\[
 \frac{1}{(1+C|s-t|)^{\frac{1}{2}}}\|v_3\|_{H^{-1}(\Gamma(s))}\le \|{\phi^s_{t}}^*v_3\|_{H^{-1}(\Gamma(t))} \le \frac{1}{(1-C|s-t|)^{\frac{1}{2}}}\|v_3\|_{H^{-1}(\Gamma(s))}\,.
\]
\end{proof}
As far as Lemma \ref{L:MatDerSmooth} is concerned, for a family of functions $\{f(t)\}_{t\in[0,T]}$, $f(t):\Gamma(t)\rightarrow \mathbb R$, one can define $\dot f$ at $\gamma=\Phi^0_t\gamma_0$ simply by $\dot f(t)[\gamma] =\phi_0^t\frac{\d}{\d t}(\phi_t^0 f(t))[\gamma_0,t]=\phi_0^t\frac{\d}{\d t}[f(t)(\Phi_t^0\gamma_0)]$. If $\{f(t)\}$ can be smoothly extended,  this is  equivalent to \eqref{E:MatDerStrong}. The following Lemmas aim at defining a weak material derivative of $f$ that translates into a weak derivative of the pull-back $\phi_t^0 f(t)$.

\begin{LemmaAndDef}\label{LD:Spaces}
Consider the disjoint union $\mathcal B_{L^2}=\bigcup_{t\in[0,T]}L^2(\Gamma(t))\times\{t\} $. The set of functions $f:[0,T]\rightarrow \mathcal B_{L^2}$, $t\mapsto (v_t,t)$ inherits a canonical vector space structure from the spaces $L^2(\Gamma(t))$ (addition and  multiplications with scalars). Given Assumption \ref{A:Phi}, for $ s\in[0,T]$ we define
$$
  L^2_{L^2(\Gamma)}:=\set{\bar v:[0,T]\rightarrow \mathcal B_{L^2}, \,t\mapsto (v_t,t)}{ \phi^s_{t}v\in L^2((0,T),L^2(\Gamma(s)))   }\,.
  $$
Abusing notation, now and in the following we identify $\bar v(t)=(v_t,t)\in  L^2_{L^2(\Gamma)}$ with $v(t)=v_t$.  Endowed with the scalar product
  $$
  \langle f,g\rangle_{ L^2_{L^2(\Gamma)}}=\mint{[0,T]}{\langle f(t),g(t)\rangle_{L^2(\Gamma(t))}}{t}\,.
  $$
  $L^2_{L^2(\Gamma)}$  becomes a Hilbert space. 
  
 In the same manner we define the space $L^2_{H^1(\Gamma)}$. For $L^2_{H^{-1}(\Gamma)}$ use ${\phi^t_{s}}^*$ instead of $\phi^s_{t}$. All three spaces do not depend on $s$.
 
 For  $\varphi\in\phi^{(\cdot)}_{s} \mathcal D((0,T),H^1(\Gamma(s)))=\set{\varphi\in L^2_{L^2(\Gamma)}}{ \phi^s_{t} \varphi\in\mathcal D((0,T),H^1(\Gamma(s))}$, it is clear how to interpret $\dot\varphi$, namely $\dot\varphi=\phi^t_s(\phi^s_t\varphi)' \in H^1(\Gamma(t))$. We say that $y\in L^2_{H^{1}(\Gamma)}$ has weak material derivative $\dot y(t)\in L^2_{H^{-1}(\Gamma)}$ iff there holds 
 \begin{equation}\label{E:WeakDerivative}
 \mint{[0,T]}{\langle \dot y,\varphi\rangle_{H^{-1}(\Gamma(t)),H^1(\Gamma(t))}}{t}=-\mint{[0,T]}{\langle y,\dot\varphi\rangle_{L^2(\Gamma(t))}}{t} -\mint{[0,T]}{\mint{\Gamma(t)}{y\varphi \div_{\Gamma} V}{\Gamma(t)}}{t}
 \end{equation}
 for all $\varphi\in\phi^{(\cdot)}_{s} \mathcal D((0,T),H^1(\Gamma(s)))$, and the definition does not depend on $s$.
 \end{LemmaAndDef}
\begin{proof}
In order to define the scalar product of $L^2_{L^2(\Gamma)}$, we must ensure measurability of $\langle f,g\rangle_{L^2(\Gamma(t))}:[0,T]\rightarrow\R$. Since $\langle f,g\rangle=\frac{1}{2}(\|f+g\|^2-\|f\|^2-\|g\|^2)$ it suffices to show measurability of $\|f\|^2_{L^2(\Gamma(t))}$ for all $f\in L^2_{L^2(\Gamma)}$. By definition of the set $L^2_{L^2(\Gamma)}$ we have $\phi_t^s f\in L^2([0,T],L^2(\Gamma(s)))$. Hence, there exists a sequence of measurable simple functions $\tilde f_n$ that converge pointwise a.e. to $\phi_t^s f$ in $L^2(\Gamma(s))$. Each $\tilde f_n$ is the finite sum of measurable single-valued functions, i.e. $\tilde f_n=\sum_{i=1}^{M_n} f_{i,n}\mathbf 1_{B_i}$, $M_n\in \mathbb N$, $f_{i,n}\in L^2(\Gamma(s))$, $[0,T]\supset B_i$ measurable and disjoint. By Lemma \ref{L:NormContinuity} the function
$$\|\phi_s^t\tilde f_n\|_{L^2(\Gamma(t))}=\sum_{i=1}^{M_n} \|\phi_s^t f_{i,n}\|_{L^2(\Gamma(t))}  \mathbf 1_{B_i}$$ 
is the finite sum of measurable functions and thus measurable. Using the continuity of the operator $\phi_s^t$, as stated in Lemma \ref{L:basics}, one infers pointwise convergence a.e.  of $\|\phi_s^t\tilde f_n\|_{L^2(\Gamma(t))}$ towards $ \|f\|_{L^2(\Gamma(t))}$ which in turn implies measurability of $ \|f\|_{L^2(\Gamma(t))}$.

Again by Lemma \ref{L:basics} we now conclude integrability of $ \|f\|_{L^2(\Gamma(t))}$  and at the same time equivalence of the norms 
$$
\left(\mint{[0,T]}{\| f\|^2_{L^2(\Gamma(t))}}{t}\right)^\frac{1}{2} \quad\text{and}\quad \left(\mint{[0,T]}{\| \phi^s_t f\|^2_{L^2(\Gamma(s))}}{t}\right)^\frac{1}{2} \,.
$$
Completeness of  $L^2_{L^2(\Gamma)}$ follows, since $ L^2_{L^2(\Gamma)}$ and $L^2((0,T),L^2(\Gamma(s)))$ are isomorph. Again because of Lemma \ref{L:basics}, $\phi^s_{t}v\in L^2((0,T),L^2(\Gamma(s)))$ is equivalent to $\phi^r_{t}v\in L^2((0,T),L^2(\Gamma(r)))$, thus the definition does not depend on the choice of $s$.
For $L^2_{H^1(\Gamma)}$ and $L^2_{H^{-1}(\Gamma)}$ we proceed similarly.

We show that the definition of the weak material derivative does not depend on $s\in[0,T]$. On $\Gamma(s)$ Equation \eqref{E:WeakDerivative} reads
\begin{equation}\label{E:WeakDerivativePullback}
 \mint{[0,T]}{\langle{ \phi_s^t}^* \dot y,\tilde\varphi\rangle_{H^{-1}(\Gamma(s)),H^1(\Gamma(s))}}{t}=-\mint{[0,T]}{\mint{\Gamma(s)}{{\left( \phi_t^sy\tilde\varphi'(t)  +\phi^s_t\left(y \div_{\Gamma(t)} V\right)\tilde\varphi\right)J^s_t}}{\Gamma(s)}}{t}
\end{equation}
for all $\tilde \varphi \in \mathcal D([0,T],H^1(\Gamma(s)))$. For $r\in[0,T]$, we now transform the relation into one on $\Gamma(r)$, using $\phi^r_s$, $(\phi^s_r)^*$ and  $\phi^r_t=\phi^r_s\circ\phi^s_t$
\[
 \mint{[0,T]}{\langle{ \phi_r^t}^* \dot y,\phi_s^r\tilde\varphi\rangle_{H^{-1}(\Gamma(r)),H^1(\Gamma(r))}}{t}=-\mint{[0,T]}{\mint{\Gamma(r)}{{\left( \phi_t^ry\left(\phi_s^r\tilde\varphi (t)\right)' +\phi^r_t\left(y  \div_{\Gamma(t)} V\right)\phi_s^r\tilde\varphi\right)J^r_t}}{\Gamma(r)}}{t},
\]
and because $\phi_s^r:H^1(\Gamma(s))\rightarrow H^1(\Gamma(r))$ is a linear homeomorphism, it also defines an isomorphism between $\mathcal D([0,T],H^1(\Gamma(s)))$ and $\mathcal D([0,T],H^1(\Gamma(r)))$.
\end{proof}
\begin{Remark}
Strictly speaking the elements of $  L^2_{X(\Gamma)}$ are equivalence classes of functions coinciding a.e. in $[0,T]$, just like the elements of $ L^2((0,T),X(\Gamma(s))) $.
\end{Remark}
The definition of the weak derivative of $y\in L^2_{H^1(\Gamma)}$ in \eqref{E:WeakDerivative} translates into weak derivatives of the pullback  $\phi^s_t y$. In order to make the connection between the two, we state the following
\begin{Lemma}\label{L:WJmult}
Let $w\in W_s(0,T)$ and $f\in C^1([0,T]\times\Gamma(s))$. Then $fw$  also lies in $W_s(0,T)$ and
\[
\left(fw\right)'=\underbrace{\partial_tfw}_{\in L^2([0,T],\Ltwo{s})}+fw'\,,
\]
where $fw'$ is to be understood as $\langle fw',\varphi\rangle_{H^{-1}(\Gamma(s)),H^1(\Gamma(s))}=\langle w', f\varphi\rangle_{H^{-1}(\Gamma(s)),H^1(\Gamma(s))}$.
\end{Lemma}
\begin{proof}
We show that for $\varphi\in\D{s}$ the function $f\varphi$ lies in $W_s(0,T)$. The claim then follows by partial integration in $\W{s}$.

\emph{1.} Because $f\in C([0,T]\times \Gamma(s))$ and the strong surface gradient  $\nabla_{\Gamma(s)} f\in \left(C([0,T]\times \Gamma(s))\right)^{n+1}$ are continuous and thus uniformly continuous on the compact set $[0,T]\times \Gamma(s)$, we infer $f \in C([0,T],C^1(\Gamma(s)))$. Note that $\textup{dist}_{[0,T]\times\Gamma(s)}((t,\gamma),(t+k,\gamma))=k$. Let $\epsilon>0$, then for sufficiently small $k_\epsilon>k>0$ one has
\[
\|f(t+k,\cdot)-f(t,\cdot)\|_\infty+\sum_{i=1}^{n+1}\|\nabla_{\Gamma(s)}^i f(t+k,\cdot)-\nabla_{\Gamma(s)}^i f(t,\cdot)\|_\infty\le \epsilon\,.
\]

\emph{2.} As to the distributional derivative of $f\varphi$, we show that $f\in C^1([0,T],C(\Gamma(s)))$.
Observe that the uniform continuity of the strong derivative $\partial_tf$ on $[0,T]\times\Gamma(s)$ allows us to estimate
\[
\|f(t+k,\cdot)-f(t,\cdot)-\partial_t f(t,\cdot)k\|_\infty=\|k\mint{[0,1]}{\partial_tf(t+\tau k,\cdot)-\partial_t f(t,\cdot)}{\tau}\|_\infty\le k\epsilon
\]
for $k_\epsilon>k>0$ sufficiently small. Again by uniform continuity of $\partial_tf$ we conclude $\partial_tf\in C([0,T],C(\Gamma(s)))$.
All told, taking into account the continuity of the pointwise multiplication 
between the respective spaces, we showed  
\[
f\varphi \in C([0,T],H^1(\Gamma(s)))\cap C^1([0,T],L^2(\Gamma(s)))\subset W_s(0,T)\,.
\]

\emph{3.} Consider now an arbitrary $w\in\W{s}$. Since $f\varphi\in\W{s}$, by partial integration as in Lemma \ref{L:WZeroT}[\emph{3.}] it follows
\[
\begin{split}
\mint{[0,T]}{\langle w',f\varphi\rangle_{\Hmone{s},\Hone{s}} }{t}=&-\mint{[0,T]}{\langle w,\left(f\varphi\right)'\rangle_{\Hone{s},\Hmone{s}} }{t}\\
=&-\mint{[0,T]}{\langle w,\partial_t f\varphi\rangle_{\Ltwo{s}} }{t}- \mint{[0,T]}{\langle w,f\varphi'\rangle_{\Ltwo{s}} }{t}\,.
\end{split}
\]
Reordering gives
\[
\mint{[0,T]}{\langle fw,\varphi'\rangle_{\Ltwo{s}} }{t}=-\mint{[0,T]}{\langle \partial_tfw+fw',\varphi\rangle_{\Hmone{s},\Hone{s}} }{t}\,
\]
for any $\varphi\in\D{s}$. Hence condition \eqref{WeakDer_equiv} holds for $fw$. Using the density property stated in Lemma \ref{L:WZeroT}[\emph{2.}], we can approximate $fw$ by continuous $H^1(\Gamma(s))$-valued functions and infer $fw\in L^2((0,T),H^1(\Gamma(s)))$. The same argument yields $\partial_t fw+f w'\in L^2((0,T),H^{-1}(\Gamma(s)))$.
\end{proof}
Finally we can define our solution space.
\begin{LemmaAndDef}\label{L:W}
The solution space $W_\Gamma$ is defined as follows
$$
W_\Gamma = \set{v\in L^2_{H^1(\Gamma)}}{\dot v\in L^2_{H^{-1}(\Gamma)}}\,.
$$ 
$W_\Gamma$ is Hilbert with the canonical scalar product $\int_0^T \langle\cdot,\cdot\rangle_{H^1(\Gamma(t))}+\langle\dot{(\cdot)},\dot{(\cdot)}\rangle_{H^{-1}(\Gamma(t))}\d t$. Also $y\in W_\Gamma$ iff $\phi^s_t y\in\W{s}$ for (every) $s\in[0,T]$. For all $\tilde \varphi \in \mathcal D((0,T),H^1(\Gamma(s)))$ there holds
\begin{equation}\label{E:DerPullback}
\mint{[0,T]}{\langle{ \phi_s^t}^* \dot y,\tilde\varphi\rangle_{H^{-1}(\Gamma(s)),H^1(\Gamma(s))}}{t}=\mint{[0,T]}{\langle( \left(\phi^s_t y\right)',J^s_t\tilde\varphi\rangle_{H^{-1}(\Gamma(s)),H^1(\Gamma(s))}}{t}\,.
\end{equation}
One has
\[
c_W\|\phi^s_t y\|_\W{s}\le\|y\|_{W_\Gamma}\le C_W\|\phi^s_t y\|_\W{s}\,,
\]
and $c_W,C_W>0$ do not depend on $s\in[0,T]$.
\end{LemmaAndDef}
\begin{proof}
For  $y\in W_\Gamma$, observe that $J^s_t \phi_t^sy\in L^2([0,T],H^1(\Gamma(s))$ and rewrite \eqref{E:WeakDerivativePullback} as
\begin{equation}\label{E:PullbackRevisited}
\mint{[0,T]}{\left\langle J^s_t \phi_t^sy,\partial_t\tilde\varphi\right\rangle_{L^2(\Gamma(s))}}{t}=- \mint{[0,T]}{\langle{ \phi_s^t}^* \dot y,\tilde\varphi\rangle_{H^{-1}(\Gamma(s)),H^1(\Gamma(s))}}{t}  -\mint{[0,T]}{\left\langle\partial_t J^s_t\phi_t^sy,\tilde\varphi\right\rangle_{L^2(\Gamma(s))}}{t}\,,
\end{equation}
for $\tilde\varphi\in  \mathcal D((0,T),H^1(\Gamma(s)))$. Hence $ J^s_t \phi_t^sy\in W_s(0,T)$, and  from Lemma \ref{L:WJmult} it follows that also $ \phi_t^sy\in W_s(0,T)$, because $\frac{1}{J^s_t}\in C^1([0,T]\times\Gamma(s))$. Note that we used $\partial_t J_t^s = \phi_t^s( \div_{\Gamma(t)} V) J_t^s$, see Lemma \ref{L:basics}. On the other hand, for any $\tilde y\in \W{s}$ one has $J^s_t \tilde y\in \W{s}$ and thus $y=\phi_s^t\tilde y\in W_\Gamma$. Hence $\phi_{(\cdot)}^s$ constitutes an isomorphism between $W_\Gamma$ and $\W{s}$.

Apply Lemma \ref{L:WJmult}  a second time to obtain $ \left(J^s_t\tilde\varphi\right)'=\partial_tJ^s_t\tilde\varphi+J^s_t\tilde\varphi'$ and because of $\tilde\varphi(0)=\tilde\varphi(T)=0\in H^1(\Gamma(s))$ by partial integration there follows from \eqref{E:PullbackRevisited}
\[
\mint{[0,T]}{\langle{ \phi_s^t}^* \dot y,\tilde\varphi\rangle_{H^{-1}(\Gamma(s)),H^1(\Gamma(s))}}{t}=\mint{[0,T]}{\langle( \left(\phi^s_t y\right)',J^s_t\tilde\varphi\rangle_{H^{-1}(\Gamma(s)),H^1(\Gamma(s))}}{t}\,,
\]
compare Lemma \ref{L:WZeroT}[\emph 3.].
This proves the second claim.

The claim of $W_\Gamma$ being Hilbert now follows. Observe that point-wise multiplication with $J^s_t$ constitutes a linear homeomorphism in $H^1(\Gamma({s}))$ whose inverse is the multiplication by $\frac{1}{J^s_t}$. One easily checks $\|J^s_t \varphi\|_{H^1(\Gamma(s))}\le c\|J^s_t\|_{C^1(\Gamma(s))}\|\varphi\|_{H^1(\Gamma(s))}\le C\|\varphi\|_{H^1(\Gamma(s))}$. This together with Lemma \ref{L:basics} yields the equivalence of the two norms on $W_\Gamma$
\[
\mint{[0,T]}{\|y\|_{H^1(\Gamma(t))}^2+\|\dot y\|_{H^{-1}(\Gamma(t))}^2}{t}\quad\textup{and}\quad \mint{[0,T]}{\|\phi^s_t y\|_{H^1(\Gamma(s))}^2+\|(\phi^s_t y)'\|_{H^{-1}(\Gamma(s))}^2}{t}\,.
\]
Completeness of $\W{s}$ then implies completeness of $W_\Gamma$.
\end{proof}
\begin{Remark}
Formula \eqref{E:DerPullback} can be seen as a generalization of the following relation. Assume $\phi^s_t y\in \mathcal D((0,T),H^1(\Gamma(s)))$. Then
\begin{equation*}
\mint{[0,T]}{\langle{ \phi_s^t}^* \dot y,\tilde\varphi\rangle_{H^{-1}(\Gamma(s)),H^1(\Gamma(s))}}{t}=\mint{[0,T]}{\langle\dot y,{ \phi_s^t} \tilde\varphi\rangle_{L^2(\Gamma(t))}}{t}\\
=\mint{[0,T]}{\langle \left(\phi^s_t y\right)',J^s_t\tilde\varphi\rangle_{L^2(\Gamma(s))}}{t}\,.
\end{equation*}
\end{Remark}
Using Lemma \ref{L:WJmult} and \ref{L:WZeroT}, it is now easy to proof
\begin{Lemma}
For two functions $v,w\in W_\Gamma$ the expression $\langle v(t),w(t)\rangle_{L^2(\Gamma(t))}$ is absolutely continuous with respect to $t\in [0,T]$ and
\[\begin{split}
\frac{1}{\d t}\mint{\Gamma(t)}{vw}{\Gamma(t)}=&\langle \dot v,w\rangle_{H^{-1}(\Gamma(t)),H^1(\Gamma(t))}+\dots\\
&\langle v,\dot w\rangle_{H^{1}(\Gamma(t)),H^{-1}(\Gamma(t))}+\mint{\Gamma(t)}{vw\div_{\Gamma(t)}V}{\Gamma(t)}\,,
\end{split}
\]
a.e. in $(0,T)$, and there holds the formula of partial integration
\[\begin{split}
\mint{[s,t]}{\langle \dot v,w\rangle_{H^{-1}(\Gamma(\tau)),H^1(\Gamma(\tau))}}{\tau} =& \langle v,w\rangle_{L^2(\Gamma(t))}- \langle v,w\rangle_{L^2(\Gamma(s))}\dots\\
& -\mint{[s,t]}{\Big[\langle v,\dot w\rangle_{H^{1}(\Gamma(\tau)),H^{-1}(\Gamma(\tau))}+\mint{\Gamma(\tau)}{vw\div_{\Gamma}V}{\Gamma(\tau)}\Big]}{\tau}\,.
\end{split}\]
\end{Lemma}

We can now formulate \eqref{E:Transport} in a weak and slightly generalized manner. Let $\tilde b\in C^1([0,T]\times\Gamma_0)$ and $b=\phi_0^t\tilde b$. We look for solutions $u\in W_\Gamma$ that satisfy $y(0)=y_0\in L^2(\Gamma_0)$ and for $f\in L^2_{H^{-1}(\Gamma)}$
\begin{equation}\label{E:WeakTransport}
\begin{split}
\frac{\textup{d}}{\textup{d}t}\mint{\Gamma(t)}{y\,\varphi}{\Gamma(t)}+\mint{\Gamma(t)}{\nabla_\Gamma y\cdot\nabla_\Gamma\varphi+by\varphi}{\Gamma(t)}= \langle\dot\varphi,y\rangle_{H^{-1}({\Gamma(t)}),H^1({\Gamma(t)})}\dots\\+\langle f,\varphi\rangle_{H^{-1}({\Gamma(t)}),H^1({\Gamma(t)})}\,,
\end{split}\end{equation}
for all $\varphi\in W_\Gamma$ and a.e. $t\in (0,T)$. One may equivalently write \eqref{E:WeakTransport} as
\begin{equation*}
\dot y+\Delta_{\Gamma(t)}y + y\left(\div_{\Gamma(t)}V+b\right)= f\quad\textup{in }H^{-1}(\Gamma(t))
\end{equation*}
for a.e. $t\in(0,T)$. 
We  apply known existence and uniqueness results for the pulled-back equation to prove 

\begin{Theorem}\label{T:UniqueSol}
Let $f\in L^2_{H^{-1}(\Gamma)}$, $y_0\in L^2(\Gamma_0)$. There exists a unique $y\in W_\Gamma$, such that \eqref{E:WeakTransport} is fulfilled for all $\phi\in W_\Gamma$ and a.e. $t\in(0,T)$. There holds
\[
\|y\|_{W_\Gamma}\le C\left(\|y_0\|_{L^2(\Gamma_0)}+\|f\|_{L^2_{H^{-1}(\Gamma)}}\right)\,.
\]
\end{Theorem}
\begin{proof}
Let us relate equation \eqref{E:WeakTransport} to the fixed domain $\Gamma(s)$ via
\[\begin{split}
\frac{\textup{d}}{\textup{d}t}\mint{\Gamma(s)}{\tilde y\,\tilde\varphi J^s_t}{\Gamma(s)}+\mint{\Gamma(s)}{\left(\sgrad{s} \tilde y (D_{\Gamma(s)}\bar\Phi^s_t)^{-1}(D_{\Gamma(s)}\bar\Phi^s_t)^{-T}\sgrad{s}\tilde\varphi +\tilde b \tilde y\tilde\varphi \right)J^s_t}{\Gamma(s)}\dots\\
=\langle{\tilde{\varphi}}',J^s_t\tilde y\rangle_{H^{-1}({\Gamma(s)}),H^1({\Gamma(s)})}+\langle\tilde f,J_t^s\tilde\varphi \rangle_{H^{-1}({\Gamma(s)}),H^1({\Gamma(s)})}\,,
\end{split}\]
with $\tilde y= \phi^s_t y$, $\tilde f=\frac{1}{J_t^s}{\phi^t_s}^*f\in L^2((0,T),H^{-1}(\Gamma(s))$ and for all $\phi^s_t\varphi=\tilde\varphi \in\W{s}$. This again is equivalent to
\[
\begin{split}
\langle \tilde y',\tilde\varphi J^s_t\rangle_{H^{-1}(\Gamma(s)),H^1(\Gamma(s))}&+\mint{\Gamma(s)}{\tilde y\,\tilde\varphi\left( \phi^s_t(\textup{div}_{\Gamma(t)} V)+\tilde b\right)J^s_t}{\Gamma(s)}+\dots\\ 
+\mint{\Gamma(s)}{\sgrad{s}& \tilde y (D_{\Gamma(s)}\bar\Phi^s_t)^{-1}(D_{\Gamma(s)}\bar\Phi^s_t)^{-T}\sgrad{s}\tilde\varphi J^s_t}{\Gamma(s)}
=\langle\tilde f,J_t^s\tilde\varphi \rangle_{H^{-1}({\Gamma(s)}),H^1({\Gamma(s)})}\,.
\end{split}
\]
With $\psi=J_t^s\tilde\varphi$ one gets for all $\psi\in\W{s}$
\begin{equation}\label{E:PulledbackEqn}
\langle \tilde y',\psi\rangle_{H^{-1}(\Gamma(s)),H^1(\Gamma(s))}+a(t,\tilde y, \psi)=\langle\tilde f,\psi \rangle_{H^{-1}({\Gamma(s)}),H^1({\Gamma(s)})}\,,
\end{equation}
with a bilinear form
\[
\begin{split}
a(t,\tilde y, \psi) =& \mint{\Gamma(s)}{\sgrad{s} \tilde y (D_{\Gamma(s)}\bar\Phi^s_t)^{-1}(D_{\Gamma(s)}\bar\Phi^s_t)^{-T}\sgrad{s}\psi}{\Gamma(s)}+\mint{\Gamma(s)}{\tilde y\left(\phi^s_t(\textup{div}_{\Gamma(t)} V)+\tilde b\right)\psi}{\Gamma(s)}\dots\\
&-\mint{\Gamma(s)}{\sgrad{s} \tilde y (D_{\Gamma(s)}\bar\Phi^s_t)^{-1}(D_{\Gamma(s)}\bar\Phi^s_t)^{-T}\sgrad{s}J^s_t\frac{\psi}{J^s_t}}{\Gamma(s)}\,.
\end{split}
\]
By Assumption \ref{A:Phi} the bilinear form $(D_{\Gamma(s)}\bar\Phi^s_t)^{-1}[\gamma](D_{\Gamma(s)}\bar\Phi^s_t)^{-T}[\gamma]$ is positive definite on the tangential space $T_\gamma\Gamma(s)$ uniformly in $s,t\in[0,T]$ and $\gamma\in\Gamma(s)$. Thus, there exists  $c>0$ such that for some $k_0\ge0$ one has $a(t,\psi,\psi)+k_0\|\psi\|_{L^2(\Gamma(s))}\ge c\|\psi\|_{H^1(\Gamma(s))}$.
We are now in the situation to apply for example \cite[Ch. III, Thrm. 1.2]{Lions1971}, to obtain a unique solution $\tilde y\in\W{s}$ to equation \eqref{E:PulledbackEqn} for initial data $\phi^s_0 y_0\in L^2(\Gamma(s))$. Moreover the solution map is continuous
\[
\|\tilde y\|_\W{s}\le C\left(\|\tilde f\|_{L^2((0,T),H^{-1}(\Gamma(s)))}+\|\phi^s_0 y_0\|_{L^2(\Gamma(s))}\right)
\]
Note again that $\|\tilde f\|_{L^2((0,T),H^{-1}(\Gamma(s)))}\le C\|f\|_{L^2_{H^{-1}(\Gamma)}}$, since the multiplication with $J_t^s$ is a globally bounded linear homeomorphism in $H^1(\Gamma(s))$, as stated in the proof os Lemma \ref{L:W}.

The transformation of \eqref{E:WeakTransport} into \eqref{E:PulledbackEqn} works both ways, hence the uniqueness of $y\in W_\Gamma$. The norms can be estimated as in Lemma \ref{L:basics} and Lemma \ref{L:W} and the theorem follows. \end{proof}

With regard to order-optimal  convergence estimates, sometimes a slightly higher regularity than $y\in W_\Gamma$ is required. Assuming $f\in L^2_{L^2(\Gamma)}$ and $y_0\in H^1(\Gamma_0)$, one can apply a  Galerkin approximation argument, see \cite[Thms.  4.4 and 4.5]{DziukElliott2007} for manifolds or \cite{Evans1998} for open sets, to obtain

\begin{equation}\label{E:Smoothness}
\| \dot y\|^2_{L^2_{L^2(\Gamma)}}+\sup_{t\in[0,T]}\|\nabla_{\Gamma(t)} y\|_{L^2(\Gamma(t))}^2 + \mint{[0,T]}{\|y\|^2_{H^2(\Gamma(t))} }{t}\le C\left(\|y\|^2_{H^1(\Gamma(0))}+ \| f \|^2_{L^2_{L^2(\Gamma)}}\right)\,.
\end{equation}
Note that from  \cite[Ch. I,Thrm. 3.1]{LionsMagenes1968} it then follows that $\phi^s_t y\in C([0,T],H^1(\Gamma(s)))$.

\section{Control constrained optimal control problems}\label{S:OCP}

Using the results from the previous section, we can now formulate all kinds of control-constrained optimal control problems known for stationary domains, see for example \cite{Troeltzsch2005}.
As a first example, given a moving surface as in Assumption \ref{A:Phi}, let $S_T:L^2_{L^2(\Gamma)}\rightarrow L^2(\Gamma(T))$ denote the solution operator $u\mapsto y(T)$, where $y$ satisfies
\begin{equation}\label{E:SolOp}
\frac{\textup{d}}{\textup{d}t}\mint{\Gamma(t)}{y\,\varphi}{\Gamma(t)}+\mint{\Gamma(t)}{\nabla_\Gamma y\cdot\nabla_\Gamma\varphi}{\Gamma(t)}= \langle\dot\varphi,y\rangle_{H^{-1}({\Gamma(t)}),H^1({\Gamma(t)})}+\langle u,\varphi\rangle_{L^2_{L^2(\Gamma)}}\,,
\end{equation}
for all $\varphi\in W_\Gamma$, and with $y(0)=0\in L^2(\Gamma_0)$.
We know, that every function $y\in W_\Gamma$ has a representation in $C([0,T],L^2(\Gamma(s)))$ for any $s\in [0,T]$, compare Lemma \ref{L:WZeroT}, and the inclusion $\phi_{(\cdot)}^s W_\Gamma\subset C([0,T],L^2(\Gamma(s)))$ is continuous (in fact compact). Thus $S_T$ is a continuous linear operator.
Consider the Control problem
\begin{equation*}
(\mathbb{P}_T)\quad \left\{\begin{array}{l}
\min_{u\in L^2_{L^2(\Gamma)}} J(u) := \frac{1}{2}\|S_T(u)-y_T\|_{L^2(\Gamma(T))}^2 + \frac{\alpha}{2} \|u\|_{L^2_{L^2(\Gamma)} }\\
\mbox{s.t. }a\le u\le b\,,
\end{array}\right.
\end{equation*}
with $\alpha,a,b\in\R$, $a<b$ ,  $\alpha>0$, and $y_T\in L^2(\Gamma(T))$. This is now a well posed problem. By standard arguments, see for example \cite[Thm. 3.15]{Troeltzsch2005}, using the weak lower semicontinuity of $J(\cdot)$, one can conclude the existence of a unique solution $u\in L^2_{L^2(\Gamma)}$.

For an other example let the linear continuous solution operator $S_d:L^2_{L^2(\Gamma)}\rightarrow L^2_{L^2(\Gamma)}$, $u\mapsto y$, where $y$ solves \eqref{E:SolOp},
and consider the problem
\begin{equation*}
(\mathbb{P}_d)\quad \left\{\begin{array}{l}
\min_{u \in L^2_{L^2(\Gamma)}} J(u) := \frac{1}{2}\|S_d(u)-y_d\|_{L^2_{L^2(\Gamma)}}^2 + \frac{\alpha}{2} \|u\|_{L^2_{L^2(\Gamma)} }\\
\mbox{s.t. }a\le u\le b\,,
\end{array}\right.
\end{equation*}
with $\alpha,a,b$ as above and $y_d\in L^2_{L^2(\Gamma)}$. Again there exists a  unique solution, see \cite[Thm. 3.16]{Troeltzsch2005}.

The first order necessary optimality condition for $(\mathbb P_d)$ reads
\begin{equation}\label{E:NecCond}
\langle S_d u-y_d, S_d (v-u)\rangle_{L^2_{L^2(\Gamma)}}+\alpha \langle u,v-u\rangle_{L^2_{L^2(\Gamma)}} =\langle \alpha u+S_d^*(S_d u-y_d),v-u\rangle_{L^2_{L^2(\Gamma)}} \ge 0\,,
\end{equation}
for all $v\in U_{\textup{ad}}=\set{v\in {L^2_{L^2(\Gamma)}}}{a\le v\le b}$. The adjoint operator $S_d^*:L^2_{L^2(\Gamma)}\rightarrow L^2_{L^2(\Gamma)}$  maps $v\in L^2_{L^2(\Gamma)}$ onto the solution $p\in W_\Gamma$ of 
 \begin{equation}\label{E:AdjSolOp}
-\langle\dot p,\varphi\rangle_{H^{-1}({\Gamma(t)}),H^1({\Gamma(t)})}+\mint{\Gamma(t)}{\nabla_\Gamma p\cdot\nabla_\Gamma\varphi}{\Gamma(t)}= \langle v,\varphi\rangle_{L^2_{L^2(\Gamma)}}\,,
\end{equation}
for all $\varphi\in W_\Gamma$, and $p(T)=0\in L^2(\Gamma(T))$. This follows if one tests \eqref{E:SolOp} with $p$ and \eqref{E:AdjSolOp} with $y$. Integrate over $[0,T]$ and use $y(0)=0$ and $p(T)=0$ to arrive at $\langle v,y\rangle_{L^2_{L^2(\Gamma)}}=\langle p,u\rangle_{L^2_{L^2(\Gamma)}}$, for $u,v\in {L^2_{L^2(\Gamma)}}$ arbitrary. 

Note that via the time transform $t' = T-t$  Equation \eqref{E:AdjSolOp} converts into equation \eqref{E:WeakTransport} with $b=-\div_{\Gamma(t)}V$. Therefore all the results from Section \ref{S:WeakSolutions} also apply to \eqref{E:AdjSolOp}.

The necessary condition \eqref{E:NecCond} characterizes the optimum $u$ as the orthogonal projection of $-\frac{1}{\alpha}S_d^*(S_d u-y_d)$ onto $U_{\textup{ad}}$. In our situation this is the pointwise application of the projection $P_{[a,b]}:\R\rightarrow [a,b]$.
\begin{Lemma}
Let $P_{U_{\textup{ad}}}$ denote the $L^2_{L^2(\Gamma)}$-orthogonal projection onto $U_{\textup{ad}}$, which is defined by
\begin{equation}\label{E:OProj}
\langle u- P_{U_{\textup{ad}}}(u),v-P_{U_{\textup{ad}}}(u)\rangle_{L^2_{L^2(\Gamma)}}\le 0\,\quad\forall v\in U_{\textup{ad}}\,.
\end{equation}
Then for $u\in L^2_{L^2(\Gamma)}$ one has for a.e. $t\in [0,T]$
\[
P_{U_{\textup{ad}}}(u)[t]=P_{[a,b]}( u[t])\,.
\]
\end{Lemma}
\begin{proof}
Because $L^2([0,T],L^2(\Gamma(s)))$ is isometrically isomorph to $L^2([0,T]\times \Gamma(s))$, and because $\phi_t^s P_{[a,b]}( u)\in L^2([0,T]\times \Gamma(s))$, we also have $\phi_t^s  P_{[a,b]}( u)\in L^2([0,T],L^2(\Gamma(s)))$.

Let $C=\set{t\in[0,T]}{P_{U_{\textup{ad}}}(u)[t]\neq P_{[a,b]}( u[t])}$ and assume $\textup{meas}(C)>0$. Now test \eqref{E:OProj} with $v= P_{[a,b]}( u[t])$ to arrive at
\begin{equation}\label{E:ProjContra}
\mint{[0,T]}{\langle u- P_{U_{\textup{ad}}}(u),P_{[a,b]}(u)-P_{U_{\textup{ad}}}(u)\rangle_{L^2(\Gamma(t))}}{t}\le 0\,.
\end{equation}
But now for a.e. $t\in[0,T]$ and a.e. $\gamma\in\Gamma(t)$ one has $$(u_t[\gamma] -P_{U_{\textup{ad}}}(u_t)[\gamma])(P_{[a,b]}(u_t)[\gamma]-P_{U_{\textup{ad}}}(u_t)[\gamma]\ge 0\,,$$
because $P_{U_{\textup{ad}}}(u_t)[\gamma]\in[a,b]$. Moreover for $t\in C$ we have 
$$(u_t[\gamma] -P_{U_{\textup{ad}}}(u_t)[\gamma])(P_{[a,b]}(u_t)[\gamma]-P_{U_{\textup{ad}}}(u_t)[\gamma]> 0\,,$$
on a set of positive measure. Since $\textup{meas}(C)>0$ this contradicts \eqref{E:ProjContra}.
\end{proof}
Introducing the adjoint state $p_d(u)=S_d^*(S_d u-y_d)$, let us now rewrite \eqref{E:NecCond} as
\begin{equation}\label{E:FONC}
u = P_{[a,b]}\left(-\frac{1}{\alpha}p_d(u)\right)\,.
\end{equation}

Similarly the unique solution u of $(\mathbb P_T)$ is characterized by $
u = P_{[a,b]}\left(-\frac{1}{\alpha}p_T(u)\right)$,
with $p_T(u)=S_T^*(S_T u-y_T)$.
Note that however the adjoint state $p_T$ in general is less smooth than $p_d$. This is because the adjoint equation, i.e. the equation describing $S_T^*:L^2(\Gamma(T))\rightarrow L^2_{L^2(\Gamma)}$, $v\mapsto p$, reads 
 \begin{equation*}
-\langle\dot p,\varphi\rangle_{H^{-1}({\Gamma(t)}),H^1({\Gamma(t)})}+\mint{\Gamma(t)}{\nabla_\Gamma p\cdot\nabla_\Gamma\varphi}{\Gamma(t)}=0\,,
\end{equation*}
for all $\varphi\in W_\Gamma$ and with $p(T)= v\in L^2(\Gamma(T))$. While Theorem \ref{T:UniqueSol} applies, this is not the case for the smoothness assertion \eqref{E:Smoothness}, as long as $y_d\in L^2(\Gamma(T))\setminus H^1(\Gamma(T))$.

Before we can discuss the discretized control problems in Section \ref{S:VarDisc}, in the next two sections we present some results on the discretization of the state equation.

\section{Finite Element Discretization}\label{S:FEDisc}


\newcommand{\Phih}[2]{\Phi^{#1}_{#2,h}}
\newcommand{\phih}[2]{\phi^{#1}_{#2,h}}
We now discretize $\Gamma$ using an approximation $\Gamma^h_0$ of $\Gamma_0$ which is globally of class $C^{0,1}$. For the sake of convenience let us assume $n=2$, i.e. $\Gamma(t)$ is a hypersurface in $\R^3$. 

Following \cite{Dziuk1988} and \cite{DziukElliott2007}, we consider $\Gamma^h_0=\bigcup_{i\in I_h}T_h^i$ consisting of triangles $T_h^i$ with corners  on $\Gamma_0$, whose maximum diameter is denoted by $h$. With FEM error bounds in mind we assume the family of triangulations $\{\Gamma^h_0\}_{h>0}$ to be regular in the usual sense that the angles of all triangles are bounded away from zero uniformly in $h$.

As detailed in \cite{DziukElliott2010_PP} and \cite{DziukElliott2007} an evolving triangulation $\Gamma^h(t)$ of $\Gamma(t)$ is obtained by subjecting the vertices of  $\Gamma^h_0$ to the flow $\bar\Phi$. Hence, the nodes of $\Gamma^h(t)$ reside on $\Gamma(t)$ for all times $t\in[0,T]$, the triangles $T_h^i$ being deformed into triangles $T_h^i(t)$ by the movement of the vertices.
Let $m_h$ denote the number of vertices $\{X_j^0\}_{j=1}^{m_h}$ in $\Gamma^h_0$. Now $X_j(t)$ solves
\begin{equation}\label{E:VertexMovement}
\frac{\d}{\d t}X_j(t)=V(X_j(t),t)\,,\quad X_j(0)=X_j^0\,.
\end{equation}

Consider the finite element space
$$V_h(t)=\set{v\in L^2({\Gamma^h(t)})}{v\in C(\Gamma^h(t))\textup{ and } \forall i\in I_h:\: v\big|_{T^i_h(t)}\in\Pi^1(T^i_h(t))}$$
of piecewise linear, globally continuous functions on $\Gamma^h(t)$, and its nodal basis functions $\{\varphi_j(t)\}_{j=1}^{m_h}$ that are one at exactly one vertex $X_i(t)$ of $\Gamma^h(t)$ and zero at all others.
 For the finite element approach, it is crucial for the triangles $T_h^i(t)$ not to degenerate while $\Gamma^h(t)$ evolves, which leads us to the following assumption.

\begin{Assumption}\label{A:Uniformness}
The angles of the triangles $T_h^i(t)$ are bounded away from zero, uniformly w.r.t. $h,i$ and $t$.  Also assume $a_t(\Gamma^h(t))=\Gamma(t)$, with the restriction of $a_t$ to $\Gamma^h(t)$ being a  homeomorphism between $\Gamma^h(t)$ and $\Gamma(t)$.
\end{Assumption}
In order to ensure optimal approximation properties of the discretization of the surface, we require $d$ to be twice Lipschitz-continuously differentiable.
\begin{Assumption}\label{A:AddReg}
$d\in C^{2,1}(\mathcal N_T)$. 
\end{Assumption}

Let us summarize some basic properties of the family $\{\Gamma^h(t)\}_{t\in[0,T]}$.
\begin{LemmaAndDef}\label{L:Phih}
 Let $\Phih{s}{\cdot}:\Gamma^h(s)\times[0,T]\rightarrow\mathbb R^3$ denote the  flow of $\Gamma^h$, i.e. the unique continuous map, such that $\Phih{s}{t}(T_h^i(s))=T_h^i(t)$ and $\Phih{s}{t}$ is affine linear on each $T_h^i(s)$. 
 
 There holds  $\Phih{r}{t}= \Phih{s}{t}\circ{\Phih{r}{s}}$ and thus $\Phih{t}{s}\circ\Phih{s}{t}=\textup{id}_{\Gamma^h(s)}$. The velocity $V_h=\partial_t \Phih{0}{t}$ is the piecewise linear interpolant of $V$ on each triangle $T_h^i(t)$.
 
 As in Lemma \ref{L:basics} we define the pull-back $\phih{s}{t}:L^2(\Gamma^h(t))\rightarrow L^2(\Gamma^h(s))$, $\phih{s}{t}v = v\circ \Phih{t}{s}$.
 
  The piecewise constant Jacobian  $J_{t,h}^s$ of $\Phih{s}{t}$  satisfies for all $s,t\in[0,T]$
\begin{equation}\label{E:Jh}
\frac{1}{C_J^h}\le \min_{\gamma\in\Gamma(s)}J^s_{t,h}(\gamma)\le \max_{\gamma\in\Gamma(s)}J^s_{t,h}(\gamma)\le C_J^h\,,
\end{equation}
for some constant $C_J^h>0$ that does not depend on $h>0$. Moreover $J_{t,h}^s$ and $\textup D_{\Gamma(s)}\Phih{s}{t}:T\Gamma(s)\rightarrow T\Gamma(t)\subset\R^3$ are differentiable with respect to time in the interior of each $T_h^i(s)$.
 
The nodal basis functions have the transport property 
\begin{equation}\label{E:MatDerAnsatz}\dot\varphi_i=\phih{t}{0}\frac{\d}{\d t}\phih{0}{t}\varphi_i\equiv0\,,\quad1\le i\le m_h\,.
\end{equation} 
Let $\nu^h(t)$ denote the normals of $\Gamma^h(t)$, defined on each $T_h^i(t)$.
\end{LemmaAndDef}

\begin{proof}

Consider a Triangle $T_h^i(s)$, $s\in[0,T]$. W.l.o.g. let $X_1(s),X_2(s),X_3(s)$ denote its vertices. Then, using matrices $X^i(t)=\left(X_2(t)-X_1(t),X_3(t)-X_1(t)\right)$, we can write $\gamma\in T_h^i(s)$ in reduced barycentric coordinates as $\lambda_\gamma(s)=(X^i(s)^TX^i(s))^{-1}X^i(s)^T(\gamma-X_1(s))$. On $T_h^i(s)$ the transformation $\Phih{s}{t}$ is uniquely defined by $\lambda_{\Phih{s}{t}\gamma}(t)=\lambda_\gamma(s)$  and thus 
\begin{equation*}
\Phih{s}{t}(\gamma)= X^i(t)(X^i(s)^TX^i(s))^{-1}X^i(s)^T(\gamma-X_1(s))+X_1(t)\,.
\end{equation*}
In the relative interior of $T_h^i(s)$ the map $\Phih{s}{t}:T_h^i(s)\rightarrow T_h^i(t)$ is differentiable and its derivative $\textup D_{T_h^i(s)}\Phih{s}{t}:\R^3\supset TT_h^i(s)\rightarrow TT_h^i(t)\subset \R^3$ can be represented in terms of the standard basis of $\R^3$ by the matrix
$
D^i_{s,t}= X^i(t)X^i(s)^TX^i(s))^{-1}X^i(s)^T\,.
$

Now one easily proves that the angle condition in Assumption \ref{A:Uniformness} ensures the existence of $c>0$ such that $\lambda^T X^i(s)^TX^i(s)\lambda\ge c \,\min(\|X_2(s)-X_1(s)\|^2,\|X_3(s)-X_1(s)\|^2)\|\lambda\|^2$ for all $\lambda\in\R^2$, $s\in[0,T]$. Hence, $\|(X^i(s)^TX^i(s))^{-1}\|_2\le \left(c \,\min(\|X_2(s)-X_1(s)\|^2,\|X_3(s)-X_1(s)\|^2)\right)^{-1}$, and since $\|X^i(s)^T\|_2^2,\|X^i(s)\|^2_2\le 2\max(\|X_2(s)-X_1(s)\|^2,\|X_3(s)-X_1(s)\|^2)$ we get
$$
\frac{\|\textup D_{T_h^i(s)}\Phih{s}{t}\d\gamma\|}{\|\d\gamma\|}\le C\frac{\max(\|X_2(s)-X_1(s)\|^2,\|X_3(s)-X_1(s)\|^2)}{\min(\|X_2(s)-X_1(s)\|^2,\|X_3(s)-X_1(s)\|^2} 
$$
for all $\d\gamma\in TT_h^i(s)$. Using again  Assumption \ref{A:Uniformness} one concludes that the quotient of edge lengths is uniformly bounded. 

Also, one easily verifies for $r,t\in[0,T]$
\begin{equation} \label{E:Invertability}
\Phih{r}{t}\gamma= (\Phih{s}{t}\circ{\Phih{r}{s}})\gamma \textup{ and }\Phih{t}{s}\Phih{s}{t}=\textup{id}_{\Gamma^h(s)}\,.
\end{equation}

 We have  $J_{t,h}^s\Big|_{T_h^i(s)}= \sqrt{\textup{det}\left( (\textup D_{T_h^i(s)}{\Phih{s}{t}})^T\textup D_{T_h^i(s)}\Phih{s}{t}\right)}$ on the triangle $T_h^i(s)$, where the derivative is represented with respect to an orthonormal basis $\mathfrak B(s)$ of $TT_h^i(s)$. 
 As per above considerations the spectral radius of  $(D_{s,t}^i)^TD_{s,t}^i$ is uniformly bounded. Hence, there exists $C_J^h> 0$ such that $J_{t,h}^s\Big|_{T_h^i(s)}=\sqrt{\det\left(\mathfrak B(s)^T(D_{s,t}^i)^TD_{s,t}^i\mathfrak B(s)\right)} \le C_J^h$. Because we can switch $s$ and $t$ and since by \eqref{E:Invertability} we have $(\Phih{s}{t})^{-1}=\Phih{t}{s}$ and thus $\frac{1}{J_{t,h}^s}=J_{s,h}^t\le C_J^h$ we conclude 
\begin{equation*}
\forall s,t\in[0,T]:\:\forall \gamma\in \Gamma_s^h\:\frac{1}{C_J^h}\le J_{t,h}^s(\gamma)\le C_J^h\,.
\end{equation*}
The trajectories $\Phih{s}{t}\gamma$, $\gamma\in\Gamma^h(s)$, the Jacobians $J_{t,h}^s$, and the entries of $D^i_{s,t}$ are differentiable for t, because the trajectories $X_j(t)$, $1\le j\le m_h$ are, compare \eqref{E:VertexMovement}. Hence also $\textup D_{\Gamma(s)}\Phih{s}{t}$ is differentiable as  a map into $\R^3$. The velocity $V_h(\gamma,s)=\partial_t\Phih{s}{t}\gamma$ equals $V$ at the vertices and depends linearly on the coordinates $\lambda_\gamma$.
As for the transport property \eqref{E:MatDerAnsatz}, it is a consequence of the piecewise linear transformations of the piecewise linear Ansatz functions $\varphi_i$ which implies $\phih{0}{t}\varphi_i(t)=\varphi_i(0)$, compare  \cite[Prop. 5.4]{DziukElliott2007}. 
\end{proof}

 \begin{Remark}
Similarly one can  prove  the map $\Phih{s}{t}:\Gamma^h(s)\rightarrow\Gamma^h(t)$ to be bi-Lipschitz with respect to the respective metrics. The Lipschitz constant $L$ does not depend on $s,t\in[0,T]$. 
 \end{Remark}
In order to compare functions defined on $\Gamma^h(t)$ with functions on $\Gamma(t)$, for sufficiently small $h>0$ we use the projection $a_t$ from \eqref{E:Projection} to lift a function $y\in L^2(\Gamma^h(t))$ to $\Gamma(t)$ 
$$
y^l(a_t(x))=y(x)\quad \forall x\in \Gamma^h(t)\,,
$$
and for $y\in L^2(\Gamma(t))$  we define the inverse lift
$$
y_l(x)=y(a_t(x))\quad \forall x\in \Gamma^h(t)\,.
$$
For small mesh parameters $h$ the lift operation $\ldown: \L2\rightarrow \Ltwoh$ defines a linear homeomorphism with inverse $\lup$. Moreover, there exists $\C>0$ such that 
\begin{equation}\label{E:IntBnd}
1-\C h^2\le \|\ldown\|_{\mathcal L(\L2,\Ld)}^2,\|\lup\|_{\mathcal L(\Ld,\L2)}^2\le 1+ \C h^2\,,\end{equation}
 as shows the following lemma.
\begin{LemmaAndDef}\label{L:Integration}
The restriction of $a_t$ to $\Gamma^h(t)$ is a piecewise diffeomorphism.
Denote by  $\delta_h$ the Jacobian of $a_t|_{\Gamma^h(t)}:\Gamma^h(t)\rightarrow\Gamma(t)$, i.e. $\delta_h=\frac{\textup{d}\Gamma}{\textup{d}\Gamma^h}=|\textup{det}(M)|$ where $M\in\R^{2\times 2}$ represents the Derivative $\textup da_t(x):T_x\Gamma^h(t)\rightarrow T_{a(x)}\Gamma(t)$ with respect to arbitrary orthonormal bases of the respective tangential space. For small $h>0$ there holds
$$
\sup_{t\in[0,T]}\sup_{\Gamma(t)} \left| 1-\delta_h\right|\le C h^2\, ,
$$
In particular $a_t|_{\Gamma^h(t)}$ is a diffeomorphism on each triangle $T_h^i(t)$. Now $\frac{1}{\delta_h}=\frac{\textup{d}\Gamma^h}{\textup{d}\Gamma}=|\textup{det}(M^{-1})|$, so that by the change of variable formula
$$
\left|\mint{\Gamma^h(t)}{v_l}{\Gamma^h(t)}-\mint{\Gamma(t)}{v}{\Gamma(t)}\right|=\left|\mint{\Gamma(t)}{v\frac{1}{\delta_h^l}-v}{\Gamma(t)}\right|\le \C h^2\|v\|_{L^1(\Gamma)}\,.
$$
Also there exists $C>0$ such that
\begin{enumerate}
 \item $\sup_{t\in[0,T]}\|\dot \delta_h(t)\|_{L^\infty(\Gamma^h(t))}\le Ch^2$, where the material derivative is to be understood in the sense of $\Phih{0}{t}$ and
 \item $\sup_{t\in[0,T]}\| \mathcal P(I-\mathcal R_h^l)\mathcal P\|_{L^\infty(\Gamma(t))}\le C h^2$, where $\mathcal R_h=\frac{1}{\delta_h^l}\left(I-d \mathcal H\right)\mathcal P^h\left(I-d \mathcal H\right)$, $\mathcal H_{ij}=\partial_{x_ix_j}d$, and $\mathcal P = \{\delta_{ij} -\nu_i\nu_j\}_{i,j=1}^{n+1}$ and $\mathcal P^h =\{ \delta_{ij} -\nu_i^h\nu_j^h\}_{i,j=1}^{n+1}$ are the  projections on the respective tangential space.
 \end{enumerate}
\end{LemmaAndDef}
\begin{proof} We summarize the proof given in \cite[Lemma 5.1]{DziukElliott2007} to extend it for the \emph{1.}  assertion. A similar proof can be found in \cite[Lemma 5.4]{DziukElliott2010_PP}.
Following  \cite{DziukElliott2007}, we use local coordinates on a triangle $e:=T_h^i(s)$ . W.l.o.g. one can assume $e\in\mathbb R^2\times \{0\}$. Since both $d$ and $\dot d=\frac{\d}{\d t}\phih{s}{t} d$ equal zero at the corners, the linear interpolates $I_hd$, $I_h\dot d$ vanish on $e$ thus, using standard finite element approximation results, we get
\[
\|d\|_{L^\infty(e)}=\|d-I_hd\|_{L^\infty(e)}\le ch^2\|d\|_{H^{2,\infty}(e)}\le c h^2\|d\|_{C^{1,1}(\mathcal N_T)}
\]
and similarly $\|\dot d\|_{L^\infty(e)}\le c h^2\|d\|_{C^{2,1}(\mathcal N_T)}$. Also  one has 
\begin{equation}\label{E:DerEst}
\|\partial_{x_i}d\|_{L^\infty(e)}\le c h \|d\|_{C^{1,1}(\mathcal N_T)}\text{ and }\|\partial_{x_i}\dot d\|_{L^\infty(e)}\le c h \| d\|_{C^{2,1}(\mathcal N_T)}
\end{equation}
 for $i=1,2$ at any point $(x_1,x_2,0)\in e$. 
 
Consider the basis $\mathfrak B(t) = \{\partial_{x_1}\Phih{s}{t},\partial_{x_2}\Phih{s}{t}, \nu^h(t)\}$ of $\R^3$, whose first two members span the tangential space  of $T_h^i(t)$. Let $(\nu_1(t),\nu_2(t),\nu_3(t))^T$ represent $\nu_l(t)=\nabla d(\cdot,t)$ with respect to $\mathfrak B(t)$. Note that $\mathfrak B(s)$ are the unit vectors. 

We have $(\nu_1(t),\nu_2(t))^T = M_t^{-1} (\textup D_{(x_1,x_2)}{\Phih{s}{t}})^T\nabla d$, with the uniformly positive definite matrix $M_t = (\textup D_{(x_1,x_2)}{\Phih{s}{t}})^T\textup D_{(x_1,x_2)}\Phih{s}{t}$. Now 
\[
\begin{split}
\mathcal O(h)=\textup D_{(x_1,x_2)}\dot d=\textup D_{(x_1,x_2)}\frac{\d}{\d t} \phih{s}{t}d=&\frac{\d}{\d t} \textup D_{(x_1,x_2)}\phih{s}{t} d = \frac{\d}{\d t}( \phih{s}{t}\nabla d^T\textup D_{(x_1,x_2)}\Phih{s}{t}) \\
=&\frac{\d}{\d t}( \phih{s}{t}(\nu_1,\nu_2)M_t) =(\dot\nu_1,\dot\nu_2)M+\underbrace{(\nu_1,\nu_2)\frac{\d}{\d t}M_t}_{\mathcal O(h)}\,,
\end{split}\]
where we used $\nu_i(\gamma,s) = \partial_{x_i} d(\gamma,s)$ and  \eqref{E:DerEst}. We subsume
  $$\|\nu_i\|_{L^\infty(e)},\,\|\dot \nu_i\|_{L^\infty(e)}\le c h\|d\|_{C^{2,1}(\mathcal N_T)}\,.$$
 One has
 \begin{equation*}
 \textup D a_t = \textup{Id}-\nabla d( \nabla d)^T - d \nabla^2 d 
 \end{equation*}
 and with $\nabla d(\cdot,s)=(\nu_1(s),\nu_2(s),\nu_3(s))^T$ we  compute (see \cite{DziukElliott2007})
 \[
 \delta_h = \|\partial_{x_1}a_t\times \partial_{x_2}a_t\|=|\nu_3| + d R(\nu,\partial_{x_1}\nu,\partial_{x_2}\nu)=\sqrt{1-\nu_{1}^2-\nu_{2}^2}+ d R(\nu,\partial_{x_1}\nu,\partial_{x_2}\nu) = 1+\mathcal O(h^2)
 \]
with some smooth remainder function $R$ and finally, since  $d= \mathcal O(h^2)$
 \[
\|\dot \delta_h \|_{L^\infty(e)}= \left\|\frac{-\nu_{1}\dot\nu_1-\nu_{2}\dot \nu_2}{\sqrt{1-\nu_{1}^2-\nu_{2}^2}} +\mathcal O(h^2)\right\|_{L^\infty(e)}\le Ch^2\,,
 \]
 where we used $|\nu_i|,|\dot\nu_i|\le C h$, $i=1,2$.
For a proof of \emph{2.} see \cite[Lemma 5.1]{DziukElliott2007}.\end{proof}
The next Lemma concerns the continuity of the lift operations between $L^2_{L^2(\Gamma^h)}$ and $L^2_{L^2(\Gamma)}$.
\begin{Lemma}\label{L:LiftContinuity}
Using the pull-back $\phih{s}{t}$ we can define $L^2_{L^2(\Gamma^h)}$ as in Lemma \ref{LD:Spaces}.
For sufficiently small $h>0$ the lift operation $\lup$ constitutes a continuous isomorphism between $L^2_{L^2(\Gamma)}$ and $L^2_{L^2(\Gamma^h)}$ with inverse $\ldown$. There holds
$$
\left|\langle f_l,g_l \rangle_{L^2_{L^2(\Gamma^h)}}-\langle f,g\rangle_{L^2_{L^2(\Gamma)}}\right|\le\C h^2|\langle f,g \rangle_{L^2_{L^2(\Gamma)}}|\,.
$$
\end{Lemma}

\begin{proof} 
Let $L^2_{L^2(T_h^i)}$, according to the flow $\Phih{s}{t}$ as defined in Lemma \ref{LD:Spaces}. 
We define $L^2_{L^2(\Gamma^h)}=\bigcup_{i\in I_h} L^2_{L^2(T_h^i)}$ with the scalar product $\int_0^T \langle \cdot,\cdot\rangle_{L^2(\Gamma^h(t))}  \d t$.

Let $\Psi_t=\Phi^t_0\circ a_t\circ\Phih{0}{t}$ denote the mapping between $\Gamma_0^h$ and $\Gamma_0$ induced by the projection $a_t$. By Assumption \ref{A:AddReg} and by the construction of $\bar\Phi^0_t$ and $\Phih{0}{t}$ is follows that $\Psi_t:\Gamma_0^h\rightarrow\Gamma_0$ is a diffeomorphism on each triangle $T_h^i(0)$ and globally one-to-one and onto. Also $\Psi_t$ and its spatial derivatives are continuous w.r.t. time $t$. 

We will show that $\bar \Psi:\Gamma^h_0\times[0,T]\rightarrow\Gamma_0\times[0,T]$, $(\gamma,t)\mapsto (\Psi_{t}(\gamma),t)$ is a piecewise diffeomorphism whose Jacobian is bounded away from zero. 
By Assumption \ref{A:Uniformness} we already have that $\bar \Psi$ is globally one-to-one. Together this implies that the pull-back with $\bar\Psi$ constitutes an isomorphism between $L^2(\Gamma_0\times[0,T])$ and $L^2(\Gamma^h_0\times[0,T])$. This again means that 
$$\phih{0}{t} f_l\in L^2([0,T],L^2(\Gamma_0^h))\Leftrightarrow\phi_t^0 f\in L^2([0,T],L^2(\Gamma_0))\,.$$

As to $\bar\Psi$ being al local diffeomorphism, the sets $\bar T_h^i=\bigcup_{t\in[0,T]}T_h^i(t)$ are a partition of $\Gamma^h_0\times[0,T]$. In the interior of each $\bar T_h^i$ the map $\bar\Psi$ is a diffeomorphism. In fact, let $\gamma\in \textup{int}(T_h^i)$ for some $1\le i\le m_h$. 
Compute 
$$\textup D_{\Gamma_0^h\times[0,T]}\bar\Psi(\gamma) =\left(\begin{array}{cc} \textup D_{\Gamma_0^h} \Psi_t(\gamma)& \partial_t \Psi_t(\gamma)\\
0&1\end{array}\right)\,.$$
 We have $ \textup D_{\Gamma_0^h}\Psi_t=\textup D_{\Gamma(t)}\Phi_0^t\textup D_{\Gamma^h(t)}a_t \textup D_{\Gamma_0^h}\Phih{0}{t}$. Its Jacobian is the product of the Jacobians $J^t_0$, $\delta_h$, and $J_{t,h}^0$ that are each bounded away from zero, uniformly in $\gamma$ and $t$, compare \eqref{E:Jh}, and the Lemmas \ref{L:Integration} and \ref{L:basics}. Hence the Jacobian of $\bar\Psi$ is bounded away from zero. 

As to continuity of $(\cdot)_l$, by Lemma \ref{L:Integration} we have that 
$$\left|\langle f_l,g_l \rangle_{L^2_{L^2(\Gamma^h)}}-\langle f,g\rangle_{L^2_{L^2(\Gamma)}}\right|= \left| \mint{[0,T]}{\mint{\Gamma(t)}{fg (\frac{1}{\delta_h^l}-1)}{\Gamma(t)}}{t} \right|\le\C h^2|\langle f,g \rangle_{L^2_{\Gamma}}|\,.$$
\end{proof}

\def\a{\mathfrak{a}}

Now, instead of dealing with Problem \eqref{E:WeakTransport} directly, w.l.o.g. we consider the equation
\begin{equation}\label{E:WeakTransport_lambda}
\frac{\textup{d}}{\textup{d}t}\mint{\Gamma(t)}{y\,\varphi}{\Gamma(t)}+\mint{\Gamma(t)}{\nabla_\Gamma y\cdot\nabla_\Gamma\varphi+\mu y\varphi}{\Gamma(t)}= \langle\dot\varphi,y\rangle_{L^2({\Gamma(t)})}+\langle f,\varphi\rangle_{L^2({\Gamma(t)})}\,,
\end{equation}
with $\bar\mu\in\mathbb R$ large enough to ensure $\mu := b+\bar \mu\ge1$. Note that $y$ solves \eqref{E:WeakTransport_lambda} iff $e^{\bar\mu t}y$ solves \eqref{E:WeakTransport} with right-hand side $e^{\bar\mu t} f$.  

In order to formulate the space-discretization of \eqref{E:WeakTransport_lambda}, consider the trial space
$$H^1_{V_h}=\set{\sum_{i=1}^{m_h}\bar y_i(t)\varphi_i(t)\in L^2_{L^2(\Gamma^h)}}{\bar y_i\in H^1([0,T])}\simeq H^1([0,T])^{m_h}\,.$$
The following definition of weak material derivatives for functions in $H^1_{V_h}$ exploits the fact that $H^1_{V_h}$ is isomorph to $H^1([0,T])^{m_h}$. It thus avoids the issue of extending the theory from Section \ref{S:WeakSolutions} for the smooth surfaces $\Gamma(t)$ to our Lipschitz approximations $\Gamma^h(t)$.
\begin{LemmaAndDef}
The weak material derivative of $v=\sum_{i=1}^{m_h}\bar v_i(t)\varphi_i(t)\in H^1_{V_h}$ is $\dot v=\phih{t}{0}(\phih{0}{t}v)'=\sum_{i=1}^{m_h}\bar v_i'(t)\varphi_i(t)$. Let further $w\in H^1_{V_h}$, then $\langle v,w\rangle_{L^2(\Gamma^h(t))}$ is absolutely continuous and
$$
\frac{\d}{\d t} \mint{\Gamma^h(t)}{vw}{\Gamma^h(t)}=\mint{\Gamma^h(t)}{\dot vw+v\dot w+vw\div_{\Gamma_h}V_h}{\Gamma^h(t)}\,.
$$
\end{LemmaAndDef}
\begin{proof}
Observe $\dot v=\phih{t}{0}(\phih{0}{t}v)'=\phih{t}{0}\left(\sum_{i=1}^{m_h}\bar v_i(t)\varphi_i(0)\right)'=\phih{t}{0}\left(\sum_{i=1}^{m_h}\bar v_i'(t)\varphi_i(0)\right)$ because $(\phih{0}{t}\varphi(t))'(\gamma)=\frac{\d}{\d t}\varphi_i(0)(\gamma)=0$ for all $\gamma\in\Gamma^h_0$, as in  \eqref{E:MatDerAnsatz}. 

Apply Lemma \ref{L:MatDerSmooth} on each triangle to see that $\langle \varphi_i(t),\varphi_j(t)\rangle_{L^2(\Gamma^h(t))}$ is smooth and $$\frac{\d}{\d t}\langle \varphi_i(t),\varphi_j(t)\rangle_{L^2(\Gamma^h(t))}=\mint{\Gamma^h(t)}{\varphi_i\varphi_j\div_{\Gamma_h}V_h}{\Gamma^h(t)}\,.$$ 
Now $$\langle v,w\rangle_{L^2(\Gamma^h(t))}=\sum_{i,j=1}^{m_h}\bar v_i(t)\bar w_j(t)\langle \varphi_i(t),\varphi_j(t)\rangle_{L^2(\Gamma^h(t))}$$ and the second assertion follows, since $\bar v_i,\bar w_j\in H^1([0,T])$, $1\le i,j\le m_h$.
\end{proof}

We approximate \eqref{E:WeakTransport_lambda} by the following semi-discrete Problem. Consider a piecewise smooth, globally Lipschitz approximation $\lambda$ of $\mu_l$, such that $\lambda\ge 1$. Find $y\in H^1_{V_h}$ such that for all $ \varphi\in H^1_{V_h}$
\begin{equation}\label{E:WeakTransport_d}
\frac{\textup{d}}{\textup{d}t}\mint{\Gamma^h(t)}{y_h\,\varphi}{\Gamma^h(t)}+\mint{\Gamma^h(t)}{\nabla_{\Gamma^h} y_h\cdot\nabla_{\Gamma^h}\varphi+\lambda y_h\varphi}{\Gamma^h(t)}= \langle\dot\varphi,y\rangle_{L^2({\Gamma^h(t)})}+\langle f_h,\varphi\rangle_{L^2({\Gamma^h(t)})}\,,
\end{equation}
and $y_h(0)=y_0^h\in V_h(0)$. One possible choice would be $\lambda = \mu_l$, $f_h=f_l$ and $y_0^h=P_0^h((y_0)_l)$ with $P_0^h$  the $L^2(\Gamma^h_0)$-orthogonal projection onto $V_h(0)$. 

First of all let us state that \eqref{E:WeakTransport_d} admits a unique solution in $H^1_{V_h}$. This is because for   $y_h=\sum_{i=1}^{m_h}\bar y_i\varphi_i$ we can rewrite \eqref{E:WeakTransport_d} as a  smooth linear ODE with non-smooth inhomogeneity  for the coefficient vector $\bar y=\{y_i\}_{i=1}^{m_h} \in H^1([0,T])^{m_h}$
\begin{equation}\label{E:MatrixTransp}
\frac{\textup{d}}{\textup{d}t} \left(M(t)\bar y_h(t)\right)+(A_\lambda(t))\bar y(t)=F(t)\,,\quad y_h(0)=y_0^h\,,
\end{equation}
with smooth mass and stiffness matrices 
$$M(t)=\{\langle \varphi_i,\varphi_j\rangle_{L^2(\Gamma^h(t))}\}_{i,j=1}^{m_h}\quad\textup{and}\quad A_\lambda(t)=\Big\{\mint{\Gamma^h(t)}{ \sgradh{t}\varphi_i \sgradh{t}\varphi_j+\lambda\varphi_i\varphi_j}{\Gamma^h(t)}\Big\}_{i,j=1}^{m_h}\,,$$
and   right-hand side $F(t)=\{\langle f_l,\varphi_i\rangle_{L^2({\Gamma^h(t)})}\}_{i=1}^{m_h}\in L^2([0,T],\mathbb R^{m_h})$, compare also  \cite{DziukElliott2007}. Observe that we used the continuity of the coefficients $\bar y_i\in H^1([0,T])$  as well as $\dot\varphi_i=0$. Existence of  a solution $\bar y_h\in H^1([0,T])^{m_h}$ of \eqref{E:MatrixTransp} can be argued by variation of constants or, more generally, one can apply an existence result by Carath\'eodory, compare \cite[Thms. 1.1+1.3]{CoddingtonLevinson1955}. Uniqueness of $y_h$ is a consequence of the following lemma.
\begin{Lemma}[Stability]\label{L:Stability}
Let $y_0\in L^2(\Gamma_0)$ and $f\in L^2_{L^2(\Gamma)}$, and let $y_h$ solve \eqref{E:WeakTransport_d} with $y_0^h\in V_h(0)$ and $f_h= f_l$. There exists $C>0$, such that for sufficiently small $h>0$ the solution satisfies
\[
\|y_h\|^2_{{L^2(\Gamma^h(T))}}+\int_0^T \mint{\Gamma^h}{\left( \nabla_{\Gamma^h} y_h\right)^2+\lambda y_h^2}{\Gamma^h(t)}\d t\le C\left(\|y_0^h\|^2_{L^2(\Gamma_0^h)}+\|f\|^2_{L^2_{L^2(\Gamma)}}\right)\,,
\]
as well as
\[
\|\dot y_h\|^2_{L^2_{L^2(\Gamma^h)}}+\textup{ess}\sup_{t\in[0,T]} \mint{\Gamma^h}{\left( \nabla_{\Gamma^h} y_h\right)^2+\lambda y_h^2}{\Gamma^h(t)}\le C\left(\|y_0^h\|^2_{H^1(\Gamma_0^h)}+\|f\|^2_{L^2_{L^2(\Gamma)}}\right)\,.
\]

\end{Lemma}
\begin{proof}
From the definition of $M$ and $A_\lambda$ using Lemma \ref{L:MatDerSmooth} on each triangle $T_h^i(t)$ there follows $M'(t)=\{\mint{\Gamma^h(t)}{ \varphi_i\varphi_j\div_{\Gamma^h(t)} V_h}{\Gamma^h(t)}\}_{i,j=1}^{m_h}$ and 
\[\begin{split}(\frac{\d}{\d t}A_\lambda)_{ij}=\mint{\Gamma^h(t)}{ -\sgradh{t}\varphi_i\left(\textup D_{\Gamma^h}V_h+\textup D_{\Gamma^h} V_h^T\right)\sgradh{t}\varphi_j&+ \dot \lambda \varphi_i\varphi_j+\dots\\
+(\sgradh{t}\varphi_i\sgradh{t}\varphi_j&+\lambda \varphi_i\varphi_j)\div_{\Gamma^h} V_h}{\Gamma^h(t)}\,.\end{split}\]

Multiply \eqref{E:MatrixTransp} by $\bar y'$ to obtain
\[\begin{split}
\overbrace{\bar y'M\bar y'}^{\|\dot y_h\|_{L^2(\Gamma^h(t))}^2}+\frac{1}{2} \frac{\textup{d}}{\textup{d}t}\left(\bar y A_\lambda\bar y\right)&=-\bar y' M'\bar y +\frac{1}{2}\bar y A_\lambda'\bar y+F\bar y'\\
&\le C  \Big(\underbrace{\|y_h\|^2_{H^1(\Gamma(t))}}_{\le \bar y A_\lambda\bar y}+\|f_l\|_{L^2(\Gamma^h(t))}^2\Big)+\frac{1}{2}\|\dot y_h\|_{L^2(\Gamma^h(t))}^2
\,,
\end{split}\]
and a Gronwall argument yields the second estimate. Multiply \eqref{E:MatrixTransp} by $\bar y$ and proceed similarly to prove the first.
\end{proof}
Obviously the material derivative depends on the evolution of the surface, i.e. different derivatives arise according to whether $\phi^s_t$ or $\phih{s}{t}$ is applied to pull back a function to a fixed domain. In order to compare $\dot z_h^l$ with $\left(\dot z_h\right)^l$ we need the following lemma.

\begin{Lemma}\label{L:MatDerTransf}
Let $y=\sum_{i=1}^{m_h}\bar y_i\varphi_i \in H^1_{V_h}$. The lift $y^l$ lies in $W_\Gamma$ with $\dot y^l\in L^2_{L^2(\Gamma)}$, and for a.e. $t\in[0,T]$ there holds
 \begin{equation*}
\left| \dot y^l - \left(\dot y\right)^l\right|\le C h^2\|\nabla_{\Gamma(t)}y^l\|\,,
 \end{equation*}
 a.e. on $\Gamma(t)$.
\end{Lemma}
\begin{proof}
We start by computing the material derivatives of  $\bar \varphi_i(x,t):\mathcal N_T\rightarrow \mathbb R$, $\bar \varphi_i(x,t)=\varphi_i^l(a_t(x),t)$, i.e. the constant extension of the trial function $\varphi_i$, $1\le i\le m_h$, along the normal field of $\Gamma(t)$,  compare the proof of \cite[Thm. 6.2]{DziukElliott2007}. Observe that $\varphi_i^l$  is not smooth along the edges of patches $a_t(T_h^j(t))$. However,  $\varphi_i^l$ is smooth  in the (relative) interior of all  $a_t(T_h^j(t))$ . 

Derive $\bar \varphi_i$ at $\gamma\in \text{relint}(T_h^j(t))$ to obtain
\begin{equation}\label{E:ExtDerivatives}
\begin{aligned}
\nabla \bar\varphi_i(\gamma,t)&=\nabla\bar\varphi_i(a_t(\gamma),t)\left(Id-\nabla d(\gamma,t)\nabla d(\gamma,t)^T-d(\gamma,t)\nabla^2d(\gamma,t)\right)\\
\partial_t\bar\varphi_i(\gamma,t)&=\partial_t\bar \varphi(a_t(\gamma),t)+\nabla\bar\varphi_i(a_t(\gamma),t)\left(-\partial_t d(\gamma,t)\nabla d(\gamma,t)-d(\gamma,t)\partial_t\nabla d(\gamma,t)\right).
\end{aligned}
\end{equation}
By construction of $\bar\varphi_i$ we have $\nabla\bar\varphi_i(a_t(\gamma))\nabla d(\gamma,t)=\nabla_\Gamma\varphi_i^l(a_t(\gamma))\nabla \nu(a_t(\gamma),t)=0$ since $\bar\varphi_i$ is constant along orthogonal lines through $\Gamma$. Also, from $d(\Phi^0_t(\gamma),t)\equiv 0$ it follows $\partial_t d =-\nabla d V$.
The (strong) material derivatives do not depend on the extension $\bar\varphi_i$, but only on the values on  $\Gamma$ and   $\Gamma^h$, respectively. One  gets $ \dot\varphi_i^l(a_t(\gamma),t)=\partial_t\bar\varphi_i(a_t(\gamma),t)+\nabla\bar\varphi_i(a_t(\gamma),t) V(a_t(\gamma),t)$ and $\dot\varphi_i(\gamma,t)=\partial_t\bar\varphi_i(\gamma,t)+\nabla\bar\varphi_i(\gamma,t) V_h(\gamma,t)$ which together with \eqref{E:ExtDerivatives} leads us to
\begin{equation}\label{E:MatDerRelation}
 \dot\varphi_i^l=\left(\dot\varphi_i \right)^l+\left(V-V_h+d((\nabla^2 d) V_h +\partial_t\nabla d)\right)\nabla_{\Gamma(t)}\varphi_i^l\,,
\end{equation} 
 in the relative interior of the patches $a_t(T_h^j(t))$, $j\in I_h$.
 
In order to prove that the pull-back $\tilde \varphi :=\phi^{0}_{t} \varphi_i^l$ lies in $ C^1([0,T],L^2(\Gamma_0))\cap C([0,T],H^1(\Gamma_0))$ for all $1\le i\le m_h$ we proceed in four steps.

\emph{1.} We show that $\tilde\varphi$ is globally Lipschitz on $\Gamma_0\times[0,T]$. Observe, that \eqref{E:ExtDerivatives} implies that all derivatives of $\tilde\varphi$ exist and are bounded on the interior of patches $P_h^i(t)=\Phi^t_0(a_t(T_h^i(t)))$. Since $\Psi_t=\Phi^t_0\circ a_t\circ\Phih{0}{t}:\Gamma^h_0\times[0,T]\rightarrow \Gamma_0$ smoothly maps the edges of $\Gamma^h_0$ into $\Gamma_0$ the domains $\bigcup_{t\in[0,T]} P_h^i(t)\times\{t\}\subset \Gamma_0\times[0,T]$ have piecewise $C^1$-boundaries. Also, $\tilde\varphi$ is continuous and we are in the situation to apply Stoke's theorem to confirm  $\tilde\varphi\in W^{1,\infty}(\Gamma_0\times[0,T])$. By Morrey's lemma, for a formulation on manifolds see \cite{MartioMiklyukovVuorinen1998}, we conclude $\tilde\varphi \in C^{0,1}(\Gamma_0\times[0,T])$.

 
\emph{2.} Now as to the time derivative, fix $\epsilon>0$ and $t\in(0,T)$. Let $L>0$ denote the global Lipschitz constant of $\tilde \varphi$ on $\Gamma_0\times[0,T]$  and choose $\eta>0$ sufficiently small such that $\sum_{i\in I_h}\textup{meas}(P_h^i\setminus P_{h,\eta}^i)\le\epsilon^2/8L^2$ where $P_{h,\eta}^i=\set{\gamma\in P_h^i}{B_\eta(\gamma)\subset P_h^i}$, the balls $B_\eta(\gamma)$ being taken with respect to the metric of $\Gamma_0$. Now, as stated above, the patches $P_h^i(t)=\Psi(t)(T_h^i)$ move continuously across $\Gamma_0$, and we can choose $K$ sufficiently small such that for all $i\in I_h$ and $k\in(-K,K)$ we have $P_{h,\eta}^i(t)\subset P_h^i(t+k) $. The derivative $\partial_t \tilde \varphi(\gamma,t)=\phi^0_t \dot\varphi^l_i$ which is defined a.e. on $\Gamma_0\times[0,T]$ then is continuous on the compact set $\mathcal K_\eta=\bigcup_{i\in I_h}\overline {P_{h,\eta}^i(t)}\times[t-K,t+K]$ and we have
\[\begin{split}
\frac{1}{k^2}\mint{\Gamma_0}{(\tilde \varphi(t+k)-\tilde \varphi(t)-\partial_t \tilde \varphi(t)k)^2}{\Gamma_0}=\frac{1}{k^2}&\sum_{i\in I_h}\bigg(\mint{P_{h,\eta}^i}{(\tilde \varphi(t+k)-\tilde \varphi(t)-\partial_t \tilde \varphi(t)k)^2}{\Gamma_0}\\
&+\mint{P_h^i\setminus P_{h,\eta}^i}{(\tilde \varphi(t+k)-\tilde \varphi(t)-\partial_t \tilde \varphi(t)k)^2}{\Gamma_0}\bigg)\,.
\end{split}\]
Substituting $\tilde \varphi(\gamma,t+k)-\tilde \varphi(\gamma,t)=\partial_t \varphi (\gamma,t) k+ \int_0^1(\partial_t \varphi(\gamma,t+\tau k)-\partial_t \tilde \varphi (\gamma,t))k\d \tau$ on $P_{h,\epsilon}^i$ like in the proof of Lemma \ref{L:WJmult} we choose $k$ small enough for 
\begin{equation}\label{E:UniformContinuity}
 \sup_{\tau\in[0,1]}\|\partial_t \varphi(t+\tau k)-\partial_t\tilde \varphi (t)\|_\infty^2\le\frac{ \epsilon^2}{2\textup{meas}(\Gamma_0)}\,,
 \end{equation}
which is possible by uniform continuity of $\partial_t\tilde\varphi$ on $\mathcal K_\eta$.  Estimating the second addend by $(2Lk)^2\sum_{i\in I_h}\textup{meas}(P_h^i\setminus P_{h,\eta}^i)\le \epsilon^2/2$ yields
$$
\limsup_{k\rightarrow 0}\frac{1}{k}\|\tilde \varphi(t+k)-\tilde \varphi(t)-\partial_t \tilde \varphi(t)k\|_{L^2(\Gamma_0)}\le \epsilon\,.
$$
 for every $\epsilon>0$. Hence $\tilde \varphi$ is differentiable into $L^2(\Gamma_0)$ with derivative $\partial_t\tilde \varphi$. 
 
 \emph{3.} Thus in order to show $\tilde\varphi\in C^1([0,T],L^2(\Gamma_0))$ it remains to prove that $\partial_t \tilde \varphi:[0,T]\rightarrow L^2(\Gamma_0)$ is continuous. 
By  \eqref{E:ExtDerivatives} $\partial_t\tilde\varphi$ is essentially bounded on $\Gamma_0\times[0,T]$. Let $M=\|\partial_t\tilde\varphi\|_{L^\infty(\Gamma_0\times[0,T])}$. For $\epsilon>0$ choose $\eta>0$ sufficiently small such that $\sum_{i\in I_h}\textup{meas}(P_h^i\setminus P_{h,\eta}^i)\le\epsilon^2/8M^2$. As above, choose $K>0$ and $\mathcal K_\eta$ accordingly. Now, choosing $k>0$ small enough such that \eqref{E:UniformContinuity} holds one arrives at
\[
\begin{split}
\|\partial_t\tilde \varphi(t+k)-\partial_t\tilde \varphi(t)\|^2_{L^2(\Gamma_0)}=\sum_{i\in I_h}\bigg(&\mint{P_{h,\eta}^i}{(\partial_t\tilde \varphi(t+k)-\partial_t\tilde \varphi(t))^2}{\Gamma_0}\dots\\
&+\mint{P_h^i\setminus P_{h,\eta}^i}{(\partial_t\tilde \varphi(t+k)-\partial_t\tilde \varphi(t))^2}{\Gamma_0}\bigg)\le\epsilon^2\,.
\end{split}
\]
 
\emph{4.} Continuity of  $\tilde \varphi:[0,T]\rightarrow H^1(\Gamma_0)$ follows similarly. In fact, the spatial partial derivatives of $\tilde\varphi$ exhibit the same piecewise smooth structure as $\partial_t\tilde\varphi$.

Finally, $\tilde \varphi = \phi^{0}_{t}\varphi_i^l\in C^1([0,T],L^2(\Gamma_0))\cap C([0,T],H^1)$ implies $\bar y_i \phi^{0}_{t}\varphi_i^l\in \W{0}$, and we conclude $y^l\in W_\Gamma$ as well as $\dot y \in L^2_{L^2(\Gamma)}$. The estimate now is a consequence of \eqref{E:MatDerRelation}.
\end{proof}
Before we proceed to the main result of this section, we need to understand the approximation of elliptic equations on $\Gamma(t)$ by finite elements on $\Gamma^h(t)$.
\begin{Lemma}\label{L:EllipticConv}
For $t\in[0,T]$ and $g\in L^2(\Gamma(t))$, $g_h\in L^2(\Gamma^h(t))$ consider 
\begin{equation}\label{E:Elliptic}
\mint{\Gamma(t)}{\nabla_\Gamma Z^g\cdot\nabla_\Gamma\varphi+\mu Z^g\varphi}{\Gamma(t)}= \langle g,\varphi\rangle_{L^2(\Gamma(t))}\,,\forall \varphi\in H^1(\Gamma(t))
\end{equation}
and
\begin{equation}\label{E:Elliptic_d}
\mint{\Gamma^h(t)}{\nabla_{\Gamma^h} Z_h^{g^h}\cdot\nabla_{\Gamma^h}\varphi+\mu_l Z_h^{g^h}\varphi}{\Gamma^h(t)}= \langle g_h,\varphi\rangle_{L^2(\Gamma^h(t))}\,,\forall \varphi\in V_h(t)
\end{equation}
with unique solutions $Z^g\in H^1(\Gamma(t))$ and $Z_h^{g^h}\in V_h(t)$. The solution operators $S(t):L^2(\Gamma(t))\rightarrow L^2(\Gamma(t))$, $g\mapsto Z^g$ and $S_h(t):L^2(\Gamma^h(t))\rightarrow V_h\subset L^2(\Gamma^h(t))$, $g_h\mapsto Z_h^{g_h}$ are self-adjoint. There exists $C$ independent of $t\in[0,T]$ such that
\begin{enumerate}
\item $\forall \varphi \in V_h(t):\:|\|\varphi^l\|_{H^1(\Gamma(t))}^2-\|\varphi\|_{H^1(\Gamma^h(t))}^2|\le Ch^2\|\varphi^l\|^2_{H^1(\Gamma(t))}<\infty$ as well as 
\item $\|(\cdot)^lS_h(t){(\cdot)^l}^*-S(t)\|_{\mathcal L(L^2(\Gamma(t)),L^2(\Gamma(t)))}\le C h^2$ and 
\item $\|(\cdot)^lS_h(t){(\cdot)^l}^*-S(t)\|_{\mathcal L(L^2(\Gamma(t)),H^{1}(\Gamma(t)))}\le C h$.
\end{enumerate}
\end{Lemma}
\begin{proof} The operators being well-defined and self-adjoint follows by standard arbuments. Assertion
\emph{1.} follows from Lemma \ref{L:Integration}[\emph{2.}], since $\varphi^l$ is continuous and piecewise smooth  on $\Gamma(t)$ and thus lies in $H^1(\Gamma(t))$ with 
\[
\mint{\Gamma^h(t)}{\|\sgradh{t}\varphi\|^2}{\Gamma^h(t)}=\mint{\Gamma(t)}{\|\sgrad{t}\varphi^l\|^2 }{\Gamma(t)}+\mint{\Gamma(t)}{\sgrad{t}\varphi^l \left(\mathcal R_h^l-\textup{Id}\right)\sgrad{t}\varphi^l}{\Gamma(t)}\,,
\]
for details see for example \cite[Lem. 5.2]{DziukElliott2007} and proof.

For a proof of \emph{2.} and \emph{3.} see \cite[Thm. 8]{Dziuk1988} and the discussion of $\ldown$ and $\lup^*$ preceding Lemma 4 in aforementioned article. The fact that $C$ does not depend on $t$ is a consequence of Assumption \ref{A:Phi} and \ref{A:Uniformness}. \end{proof}
\begin{Theorem}\label{L:SpaceConv}
Let Assumption \ref{A:Phi}, \ref{A:Uniformness} and \ref{A:AddReg} hold and let $y\in W_\Gamma$ solve \eqref{E:WeakTransport_lambda}  for some $f\in L^2_{L^2(\Gamma)}$, $y_0\in H^1(\Gamma_0)$, such that \eqref{E:Smoothness} holds. Let $y_h$ solve \eqref{E:WeakTransport_d} with $\lambda = \mu_l$ and $f_h= f_l$ and some approximation $y_0^h$ of $(y_0)_l$. There exists $C>0$ independent of $y$ and $h$ such that
\[
\|y_h^l-y\|^2_{L^2_{L^2(\Gamma)}}\le  C\left(\|y_h(0)-y_l(0)\|^2_{{L^2(\Gamma^h_0)}} +h^4\left(\|y_0\|_{H^1(\Gamma_0)}^2+\|y_0^h\|_{H^1(\Gamma_0^h)}^2+\|f\|^2_{L^2_{L^2(\Gamma)}}\right)\right)\,.
\]
\end{Theorem}
\begin{proof}
Define $z=S(t)\left(y_h^l-y\right)$ and $z_h=S_h(t)\left(\delta_h\left(y_h-y_l\right)\right)$ with $S(t)$ and $S_h(t)$ as in Lemma \ref{L:EllipticConv}. Now $\delta_h\left(y_h-y_l\right)=\lup^*\left(y^l_h-y\right)$ and hence it follows  from Lemma \ref{L:EllipticConv}[\emph{2.}] that 
\begin{equation}\label{E:ZConv}
\|z_h^l-z\|_{L^2(\Gamma(t))}= \|(\lup S_h\lup^*-S)(y_h^l-y) \|_{L^2(\Gamma(t))}\le Ch^2\|y_h^l-y\|_{L^2(\Gamma(t))}\,,
\end{equation}
 Observe now for $z_h=\sum_{i=1}^{m_h}\bar z_{i}\varphi_i$ using Lemma \ref{L:MatDerTransf} we get
$$
Y=\{ \langle{y_h^l-y,\varphi_i^l}\rangle_{L^2(\Gamma(t))}\}_{i=1}^{m_h}\in H^1([0,T] )^{m_h}\,,\textup{ and thus }\bar z= (A_\lambda)^{-1}Y\in H^1([0,T] )^{m_h}\,.
$$
Hence  $\bar z\in H^1_{V_h}$ and again by Lemma \ref{L:MatDerTransf} $z_h^l\in W_\Gamma$ as well as $\dot z_h^l(t)\in L^2(\Gamma(t))$.

We can now test \eqref{E:WeakTransport_lambda} with $z_h^l$, using \eqref{E:Elliptic} in the process, to obtain
\begin{equation}\label{E:step1}
\begin{split}
\frac{\textup{d}}{\textup{d}t}\langle{y,z_h^l}\rangle_{L^2(\Gamma(t))}+\langle{y, y_h^l-y}\rangle_{L^2(\Gamma(t))}= \langle\dot{z_h^l},y\rangle_{L^2(\Gamma(t))}
+\langle f,z_h^l\rangle_{L^2({\Gamma(t)})}\dots\\+\langle{-\Delta_\Gamma y+\mu y, z-z_h^l}\rangle_{L^2(\Gamma(t))}\,,
\end{split} \end{equation}
and testing \eqref{E:WeakTransport_d} with $z_h$ gives
\begin{equation}\label{E:Step1.5}
\frac{\textup{d}}{\textup{d}t}\langle{y_h,z_h}\rangle_{L^2(\Gamma^h(t))}+\langle{y_h^l, y_h^l-y}\rangle_{L^2(\Gamma(t))}=\langle\dot z_h,y_h\rangle_{L^2(\Gamma^h(t))}+\langle f_l,z_h\rangle_{L^2({\Gamma^h(t)})}\,.
\end{equation}
Now, since the strong material derivative $\dot \delta_h$ exists and is continuous on each triangle $T^i_h(t)$, the scalar products  $\langle \varphi_i,\varphi_j\delta_h\rangle_{L^2(\Gamma^h(t))}$, $1\le i,j\le m_h$, are differentiable with $$\frac{\d}{\d t}\langle \varphi_i,\varphi_j\delta_h\rangle_{L^2(\Gamma^h(t))}=\mint{\Gamma^h(t)}{\delta_h\varphi_i\varphi_j\div_{\Gamma^h}V_h+\dot \delta_h\varphi_i\varphi_j}{\Gamma^h(t)}$$ and we have
\[
\begin{aligned}
\frac{\textup{d}}{\textup{d}t}\langle{y_h^l,z_h^l}\rangle_{L^2(\Gamma(t))}=&\frac{\textup{d}}{\textup{d}t}\langle{y_h,z_h\delta_h}\rangle_{L^2(\Gamma^h(t))}\\
=&\frac{\textup{d}}{\textup{d}t}\langle{y_h,z_h}\rangle_{L^2(\Gamma^h(t))}+\langle y_h,\dot z_h(\delta_h-1)\rangle_{L^2(\Gamma^h(t))}+\langle y_h,z_h\dot\delta_h\rangle_{L^2(\Gamma^h(t))}\dots\\
&+\langle \dot y_h,z_h(\delta_h-1)\rangle_{L^2(\Gamma^h(t))}+\langle y_h,z_h\div_{\Gamma^h}V_h(\delta_h-1)\rangle_{L^2(\Gamma^h(t))}\,.
\end{aligned}
\]
Hence, we can rewrite \eqref{E:Step1.5} by means of the $L^2(\Gamma(t))$
\begin{equation}\label{E:Step2}
\frac{\textup{d}}{\textup{d}t}  \langle{y_h^l,z_h^l}\rangle_{L^2(\Gamma(t))}+\langle{y_h^l, y_h^l-y}\rangle_{L^2(\Gamma(t))}=\langle(\dot z_h)^l,y_h^l\rangle_{L^2(\Gamma(t))}+\langle f,z_h^l\rangle_{L^2({\Gamma(t)})}+R^h\,,
\end{equation}
with
\[
\begin{aligned}
R^h=&\langle y_h,z_h\dot\delta_h\rangle_{L^2(\Gamma^h(t))}+\langle \dot y_h,z_h(\delta_h-1)\rangle_{L^2(\Gamma^h(t))}+\langle y_h,z_h\div_{\Gamma^h}V_h(\delta_h-1)\rangle_{L^2(\Gamma^h(t))}\dots\\
&+\langle f_l,z_h(1-\delta_h)\rangle_{L^2({\Gamma^h(t)})}\,.
\end{aligned}
\]
Subtracting \eqref{E:step1} from \eqref{E:Step2} yields
\[\begin{split}
\frac{\textup{d}}{\textup{d}t}\langle{y_h^l-y,z_h^l}\rangle_{L^2(\Gamma(t))}+\|{y_h^l-y}\|^2_{L^2(\Gamma(t))}=\langle(\dot z_h)^l-\dot z_h^l,y\rangle_{L^2(\Gamma(t))}+{\langle\dot z_h,(y_h-y_l)\delta_h\rangle_{L^2(\Gamma^h(t))}}\dots\\
+R^h+\langle{-\Delta_\Gamma y+\mu y, z_h^l-z}\rangle_{L^2(\Gamma(t))} \,.
\end{split}\]
From \eqref{E:Elliptic_d} we know ${\langle\dot z_h,\delta_h(y_h-y_l)\rangle_{L^2(\Gamma^h(t))}}=\bar z_h' A_\lambda\bar z_h= \frac{1}{2}\frac{\textup{d}}{\textup{d}t}(\bar z_h A_\lambda\bar z_h)-\frac{1}{2}\bar z_h A_\lambda'(t)\bar z_h$, in the notation of \eqref{E:MatrixTransp}. Now, using \eqref{E:ZConv} and
$$
|R^h|\le C h^2\|z_h\|_{L^2(\Gamma^h(t))}\left(\|y_h\|_{L^2(\Gamma^h(t))}+\|\dot y_h\|_{L^2(\Gamma^h(t))}+\|f_l\|_{L^2(\Gamma^h(t))}\right)
$$
we can estimate
\[\begin{split}
\frac{1}{2}\frac{\textup{d}}{\textup{d}t}(\bar z_h A_\lambda(t)\bar z_h)+\|{y_h^l-y}\|^2_{L^2(\Gamma(t))}\le& C\Big(h^2 \|y\|_{L^2(\Gamma(t))}\|\nabla_{\Gamma^h(t)}z_h\|_{(L^2(\Gamma^h(t)))^{n=1}}\dots\\
+{\| z_h\|^2_{H^1(\Gamma^h(t))}}&
+h^2\|y\|_{H^2(\Gamma(t))}\|y_h-y_l\|_{L^2(\Gamma^h(t))} \Big)+|R^h|\\
\le& \frac{1}{2}\|y_h-y_l\|_{L^2(\Gamma^h(t))}^2+ C\bigg(\bar z_h A_\lambda(t) \bar z_h \dots\\
+h^4\Big(\|y_h\|^2_{L^2(\Gamma^h(t))}&+ \|\dot y_h\|^2_{L^2(\Gamma^h(t))}+ \|f_l\|_{L^2(\Gamma^h(t))}^2+\|y\|_{H^2(\Gamma(t))}^2\Big) \bigg)\,.
\end{split}\]
We can now apply Gronwall's lemma for
\begin{equation}\label{E:AfterGronwall}
\begin{aligned}
\left[\bar z_h A_\lambda(t) \bar z_h\right]_0^T&+\mint{[0,T]}{\|{y_h^l-y}\|^2_{L^2(\Gamma(t))}}{t}\\
\le Ch^4&\mint{[0,T]}{\|y_h\|^2_{L^2(\Gamma^h(t))}+ \|\dot y_h\|^2_{L^2(\Gamma^h(t))}+ \|f_l\|_{L^2(\Gamma^h(t))}^2+\|y\|_{H^2(\Gamma(t))}^2}{t}\,,
\end{aligned}
\end{equation}
and with the stability estimate \eqref{E:Smoothness}  and the Lemmas \ref{L:Stability}  and \ref{L:LiftContinuity} we finally arrive at 
\begin{equation}\label{E:LastStep}
\begin{aligned}
\mint{[0,T]}{\|{y_h^l-y}\|^2_{L^2(\Gamma(t))}}{t}\le& C\Bigg(\overbrace{\mint{\Gamma^h_0}{\left(\nabla_{\Gamma^h_0} z_h\right)^2+\lambda z_h^2}{\Gamma^h_0}}^{=\langle y_h^l(0)-y(0),z_h\rangle_{L^2(\Gamma_0)}}\dots\\
&+ h^4 \left(\|y_0\|_{H^1(\Gamma_0)}^2+\|y_0^h\|_{H^1(\Gamma_0^h)}^2+\|f\|^2_{L^2_{L^2(\Gamma)}}\right)\Bigg)\,.
\end{aligned}
\end{equation}
Apply again \eqref{E:ZConv} to prove the lemma.
\end{proof}
\begin{Remark}\label{R:Projection}
Depending on the regularity of $y_0$, possible choices of $y_0^h$ yielding $\mathcal O(h^2)$-convergence of $y_h^l$ comprehend the piecewise interpolation of $(y_0)_l$ and the $L^2(\Gamma_0)$-orthogonal projection of $(y_0)_l$ onto $V_h(0)$. For the latter, the term involving $z_h$ in \eqref{E:LastStep} vanishes completely, but it's  $H^1(\Gamma_0)$-stability requires further investigation.
\end{Remark}
The order of convergence is lower, if the solution of \eqref{E:WeakTransport_lambda}  does not satisfy the additional regularity estimate \eqref{E:Smoothness}.
\begin{Theorem}\label{L:WeakSpaceConv}
Let Assumption \ref{A:Phi}, \ref{A:Uniformness}  and \ref{A:AddReg} hold and let $y\in W_\Gamma$ solve \eqref{E:WeakTransport_lambda}  for  $f\equiv 0$, and $y_0\in L^2(\Gamma_0)$. There exists $C>0$ independent of $y$ and $h$ such that for the solution $y_h$ of \eqref{E:WeakTransport_d} with $y_0^h=P_0^h((y_0)_l)$ and $f_h\equiv 0$ there holds
\[
\|y_h^l-y\|^2_{L^2_{L^2(\Gamma)}}\le  C \left(h^2+\sup_{t\in[0,T]}\|\lambda^l-\mu\|^2_{L^\infty(\Gamma(t))}\right)\|y_0\|^2_{L^2(\Gamma_0)}\,.
\]
\end{Theorem}

\begin{proof}
We proceed as in the proof of Theorem \ref{L:SpaceConv} up to \eqref{E:step1} which now reads
\begin{equation*}
\frac{\textup{d}}{\textup{d}t}\langle{y,z_h^l}\rangle_{L^2(\Gamma(t))}+\langle{y, y_h^l-y}\rangle_{L^2(\Gamma(t))}= \langle\dot{z_h^l},y\rangle_{L^2(\Gamma(t))}
+\langle{-\Delta_\Gamma y+\mu y, z-z_h^l}\rangle_{H^{-1}(\Gamma(t)),H^1(\Gamma(t))}\,,
\end{equation*}
Analogously to \eqref{E:ZConv} we can apply Lemma \ref{L:EllipticConv}[\emph{3.}] and estimate the last term through
\[\begin{split}
|\langle{-\Delta_\Gamma y+\lambda y, z_h^l-z}\rangle_{H^{-1}(\Gamma(t)),H^{1}(\Gamma(t))}|&\le \|-\Delta_\Gamma y+\lambda y\|_{H^{-1}(\Gamma(t))}\|z_h^l-z\|_{H^1(\Gamma(t))}\dots\\ 
&\le  \|-\Delta_\Gamma y+\lambda y\|_{H^{-1}(\Gamma(t))}Ch\|y_h^l-y\|_{L^2(\Gamma(t))}\,.
\end{split}\]
On the other hand \eqref{E:Step2} becomes
\[
\frac{\textup{d}}{\textup{d}t}  \langle{y_h^l,z_h^l}\rangle_{L^2(\Gamma(t))}+\langle{y_h^l, y_h^l-y}\rangle_{L^2(\Gamma(t))}=\langle(\dot z_h)^l,y_h^l\rangle_{L^2(\Gamma(t))}+\langle(\mu_l-\lambda)y_h,z_h\rangle_{L^2(\Gamma^h(t))}+R^h\,.
\]
Continue as in the proof of Theorem \ref{L:SpaceConv} to finally arrive at the analogue of \eqref{E:AfterGronwall}
\[
\begin{aligned}
\left[\bar z_h A_\lambda(t) \bar z_h\right]_0^T+\mint{[0,T]}{\|{y_h^l-y}\|^2_{L^2(\Gamma(t))}}{t}\le &C(h^4+\sup_{t\in[0,T]}\|\lambda^l-\mu\|^2_{L^\infty(\Gamma(t))})\mint{[0,T]}{\|y_h\|^2_{L^2(\Gamma^h(t))}}{t}\dots\\
&+ Ch^2 \mint{[0,T]}{\|y\|_{H^1(\Gamma(t))}^2+ h^2\|\dot y_h\|^2_{L^2(\Gamma^h(t))}}{t}\,,
\end{aligned}
\]
Note that due to Lemma \ref{L:Integration}
$$|\bar z_h(0) A_\lambda(0) \bar z_h(0)|=|\langle y_h^l(0)-y(0),z_h\rangle_{L^2(\Gamma_0)}|\le\overbrace{|\langle y_h(0)-y_l(0),z_h\rangle_{L^2(\Gamma_0^h)}|}^{=0\text{ since }y_0^h=P_0^h((y_0)_l)}+Ch^2\|y_0\|^2_{L^2(\Gamma_0)}\,.$$

In view of  Lemma \ref{L:Stability}  it remains to bound $\int_0^Th^2\|\dot y_h\|^2_{L^2(\Gamma^h(t))}\d t$. Again thanks to Lemma \ref{L:Stability} we  have 
\begin{equation*}\int_0^T \|\dot y_h\|^2_{L^2(\Gamma^h(t))} \d t\le C\|y_0^h\|^2_{H^1(\Gamma_0)}\,.\end{equation*}
But an inverse estimate, compare for example \cite[Thm. 17.2]{CiarletLions1991}, yields $\|y_0^h\|_{H^{1}(\Gamma^h_0)}\le \frac{C}{h}\|y_0^h\|_{L^2(\Gamma^h_0)}$, and because of the continuity of the lift $\ldown$ and of the $L^2$-projection $P_0^h$ the theorem follows.
\end{proof}

\section{Implicit Euler discretization}\label{S:DGDisc}

In order to solve \eqref{E:WeakTransport} we apply a vertical method of lines. The time discretization is carried out by discontinuous Galerkin -- implicit Euler discretization in $L^2_{L^2(\Gamma^h)}$.
\def\Vk{W_k^h}
For $N\in\mathbb N$, consider an equidistant partition $I_n = (t_{n-1},t_n]$ of $[0,T]$ with $1\le n\le N$, $k=\frac{T}{N}$ and $t_n=kn$. The trial space for the discontinuous Galerkin method (DGM)  is the space of 'piecewise constant' functions
$$\Vk=\set{v\in L^2_{L^2(\Gamma^h)}}{\forall 1\le n\le N:\:\exists v^n\in V_h(t_n):\: v\equiv \phih{t}{t_n}v^n \textup{ on }I_n}\,.$$
Note that  in the following we  will omit the operators $\phih{s}{t}$ when dealing with functions $w\in \Vk$. 
Also, to further simplify notation let $\a(t;\psi,\varphi)=\mint{\Gamma^h(t)}{\nabla_{\Gamma^h} \psi\cdot\nabla_{\Gamma^h}\varphi+\lambda \psi\varphi}{\Gamma^h(t)}$ as well as $\langle\cdot,\cdot\rangle_n= \langle\cdot,\cdot\rangle_{L^2(\Gamma^h(t_n))}$. 
W.l.o.g. we temporarily assume 
\begin{equation}\label{E:big_lambda}
\inf_{t\in[0,T],\gamma\in\Gamma^h(t)}\lambda(\gamma,t)>M+2\,,
\end{equation}
with  $M=\sup_{\tau\in[0,T]} \|\div_{\Gamma^h(\tau)}V_h\|_{L^\infty(\Gamma^h(\tau))}$ such that 
$$
\a(t;\varphi,\varphi) - M\|\varphi\|_{L^2(\Gamma^h(t))}^2\ge  \|\varphi\|_{H^1(\Gamma^h(t))}^2+\|\varphi\|_{L^2(\Gamma^h(t))}^2
$$
for all $t\in[0,T]$, $h>0$ and all $\varphi\in H^1(\Gamma^h(t))$.

To motivate the DGM insert the Ansatz $y_h^k(t)=\sum_{n=1}^N \phih{t}{t_n} (y^n\mathbf 1_{I_n})\in \Vk$ with $y^n\in V_h(t_n)$ into \eqref{E:WeakTransport_d}. 

If one understands the time-derivative in \eqref{E:WeakTransport_d} in a distributional sense, integration over time formally yields
\[
\langle y^n-y^{n-1},\varphi\rangle_{n-1} + \mint{I_n}{ \a(t;y^n,\varphi)+\langle y^n\div_{\Gamma^h} V_h,\varphi\rangle_{L^2(\Gamma^h(t))}}{t}=\mint{I_n}{\langle f_h,\varphi\rangle_{L^2(\Gamma^h(t))}}{t}\,,
\]
for  smooth test functions $\varphi$.  Instead, apply test functions $ \varphi\in \Vk$
and use $\dot y^n = \dot \varphi^n=0$  to obtain 
$$\mint{I_n}{ \langle y^n\div_{\Gamma^h} V_h,\varphi\rangle_{L^2(\Gamma^h(t))}}{t}= \langle y^n,\varphi^n\rangle_{n}-\langle y^n,\varphi^n\rangle_{n-1}\,.$$
Finally, to arrive at a computable scheme, lump the Integral over $\a(t,\cdot,\cdot)$ and replace the right-hand side appropriately. For arbitrary parameters $y_0^h\in V_h(0)$ and $f_h\in L^2_{L^2(\Gamma^h)}$ we rewrite the scheme as
\begin{equation}\label{E:FullyDisc}
\begin{split}
y_f^0=y_0^h\,,\quad&\forall \varphi\in \Vk\,,\;1\le n\le N\,:\\
&\langle y_f^n,\varphi\rangle_n -\langle y_f^{n-1},\varphi\rangle_{n-1} + k \a_n(y_f^n,\varphi)=\mint{I_n}{\langle \phih{t_n}{t}f_h,\varphi\rangle_{n}}{t}\,,
\end{split}
\end{equation}
\def\r{{\mathfrak r}}
where  $y_0^h$, $f_h$, and $\lambda$ are the same as in \eqref{E:WeakTransport_d}. For the approximation of the integral $\a_n$ we assume $\a_n(\psi,\varphi)=\a(t_n;\phih{t_n}{t}\psi,\phih{t_n}{t}\varphi)+\r_n(\psi,\varphi)$, with a remainder 
\begin{equation}\label{E:r}
|\r_n(\psi,\varphi)|\le C_\r k  \|\psi\|_{H^1(\Gamma^h(t_n))}\|\varphi\|_{H^1(\Gamma^h(t_n))}\,.
\end{equation}
One possible choice is $\r_n\equiv  0$ for $1\le n\le N$, but when it comes to approximating an adjoint equation such as \eqref{E:AdjSolOp} we will want to choose $\r$ more freely.
In order to proof convergence of the scheme \eqref{E:FullyDisc} in $L^2_{L^2(\Gamma^h)}$ we make use of  stability properties of the adjoint scheme
\begin{equation}\label{S:DiscrAdj}
\begin{split}
z_g^{N+1}=z_T\,,\quad&\forall \varphi\in \Vk\,,\;1\le n\le N\,:\\
&\langle z_g^n,\varphi\rangle_n -\langle z_g^{n+1},\varphi\rangle_{n} + k \a_n(\varphi,z_g^n)=\mint{I_n}{\langle \phih{t_n}{t}g_h,\varphi\rangle_{n}}{t}\,.
\end{split}
\end{equation}
with $g_h\in L^2_{L^2(\Gamma^h)}$, $z_T\in V_h(T)$.
In Section \ref{S:VarDisc} it will be important that given snapshots $\{\Gamma^h(t_n)\}_{n=1}^N$ of the surface \eqref{E:FullyDisc} and \eqref{S:DiscrAdj} can be evaluated exactly for certain right-hand sides $f_h$ and $g_h$, e.g. $g_h\in \Vk$. Let us introduce the mean value of a function  $y\in L^2_{L^2(\Gamma^h)}$ over an interval $I_n$. 
\begin{LemmaAndDef}\label{LD:Mean}
Let $\phih{s}{t}$ denote the pullback operator associated to the flow $\Phih{s}{t}$ as in Lemma \ref{L:basics} and let $s\in[0,T]$. The mean value of a function $y\in L^2_{L^2(\Gamma^h)}$ is defined  as $\bar y^n(s) =\frac{1}{k} \mint{I_n}{\phih{s}{t} y}{t}$ for $t\in I_n$. Because 
$$\mint{I_n}{\phih{s}{t} y}{t}=\mint{I_n}{\phih{s}{r}\phih{r}{t} y}{t}=\phih{s}{r}\underbrace{\mint{I_n}{\phih{r}{t} y}{t}}_{\bar y_n(r)}\,,$$
$\bar y^n$ does not depend on $s\in[0,T]$.
\end{LemmaAndDef}

Similarly one could define the mean value of $y\in W_\Gamma$ if one were to investigate a horizontal method-of-lines approach.

Now for $y_0\equiv0$, $z_T\equiv0$ the schemes are adjoint in the sense 
$$k\sum_{n=1}^N\langle \bar f_h^n,z_g\rangle_{n}=k\sum_{n=1}^N\langle \bar g_h^n,y_f\rangle_{n}\,,$$
i.e. the discrete solution operators $f_h\mapsto y_f$ and $g_h\mapsto z_g$ are adjoint as operators from $(L^2_{L^2(\Gamma^h)},\langle\cdot,\cdot\rangle_{h,k})$ into itself, where $L^2_{L^2(\Gamma^h)}$ is equipped with the scalar product 
\begin{equation}\label{E:DiscSP}
\langle f,g\rangle_{h,k}=k\sum_{n=1}^N\mint{I_n}{\langle (\phih{t_n}{t}f),(\phih{t_n}{t}g)\rangle_{n}}{t}\,.
\end{equation}
\begin{Lemma}\label{L:ScalarProduct}
Let $\|\cdot\|_{h,k}$ denote the norm induced by $\langle \cdot,\cdot\rangle_{h,k}$. The norms $\|\cdot\|_{L^2_{L^2(\Gamma^h)}}$ and $\|\cdot\|_{h,k}$ on $L^2_{L^2(\Gamma^h)}$ are equivalent and there holds
\[
\left|\langle f,g\rangle_{h,k}-\langle f,g\rangle_{L^2_{L^2(\Gamma^h)}}\right|\le Ck \left|\langle f,g\rangle_{L^2_{L^2(\Gamma^h)}}\right|\,.
\]
\end{Lemma}
\begin{proof} The result follows from the identity
\[
\mint{[0,T]}{\mint{\Gamma^h(t)}{fg}{\Gamma^h(t)}}{t}=\sum_{n=1}^N\mint{I_n}{\mint{\Gamma^h(t_n)}{(\phih{t_n}{t}f)(\phih{t_n}{t}g)J_{t_n,h}^{t}}{\Gamma^h(t_n)}}{t}\,,
\]
and $J_t^{t_n}$ being Lipschitz with $J_{t_n,h}^{t_n}\equiv 1$.
\end{proof}
\def\M{M}
Note also that for $z\in \Vk$, since $\dot z^n=0$ on $I_n$, we can apply the mean value theorem to obtain for some $t\in I_n$
\begin{equation}\label{E:L2NormVariation_pre}
|\|z^n\|^2_{L^2(\Gamma^h(t))}-\|z^{n}\|^2_{n}| =k |\langle z^n\div_{\Gamma^h(\Theta_n)} V_h,z^n\rangle_{L^2(\Gamma^h(\Theta_n))}|\le \M k\|z^n\|^2_{L^2(\Gamma^h(\Theta_n))}
\end{equation}
with $\Theta_n\in (t,t_n)$. Apply \eqref{E:L2NormVariation_pre} to itself to obtain for some $\tilde\Theta_n\in(\Theta_n,t_n)$
\begin{equation}\label{E:L2NormVariation}
\begin{split}
|\|z^n\|^2_{L^2(\Gamma^h(t))}-\|z^{n}\|^2_{n}|&\le \M k\left(\|z^{n}\|^2_{n}+\left(\|z^n\|^2_{L^2(\Gamma^h(\Theta_n))}-\|z^{n}\|^2_{n}\right)\right)\\
&\le  \M k\left(\|z^n\|^2_{n}+\M k \|z^n\|^2_{L^2(\Gamma^h(\tilde\Theta_n))}\right)\\
&\le\M k\left(1+C_{L^2(\Gamma^h)}\M k \right)\|z^{n}\|^2_{n}\,.
\end{split}
\end{equation}
A similar continuity assertion holds for the $H^1(\Gamma^h(t))$-norm, as shows the following lemma.
\begin{Lemma}\label{L:H1cont}
Let $y,z \in H^1_{V_h}$,  $\tilde \lambda\in C(\Gamma^h(s)\times[0,T])$, and $\lambda = \phih{t}{s} \tilde \lambda$. There exists $C>0$ such that for every $s\in I_n$ 
\[
\left| \int_{I_n}\a(s; \phih{s}{t}y, \phih{s}{t}z)\d t-\int_{I_n}\a(t;y,z)\d t\right|\le C k \int_{I_n}\|\phih{s}{t}y\|_{H^1(\Gamma^h(s))}\|\phih{s}{t}z\|_{H^1(\Gamma^h(s))}\d t\,,
\]
i.e. for $z\in \Vk$ we have
\[
\left| k\a(s; \bar y^n, z^n)-\int_{I_n}\a(t;y,z)\d t\right|\le C k \int_{I_n}\|\phih{s}{t}y\|_{H^1(\Gamma^h(s))}\|z^n\|_{H^1(\Gamma^h(s))}\d t\,.
\]
In particular with $\lambda\equiv 1$ the estimates hold for $\a(t;\varphi,\varphi)=\|\varphi\|_{H^1(\Gamma^h(t))}^2$.
\end{Lemma}
\begin{proof}\def\DS{\tilde\Delta}
We abbreviate $\DS(s,t)=D_{\Gamma^h(s)}\Phih{s}{t}(D_{\Gamma^h(s)}\Phih{s}{t})^TJ_{t,h}^s$. Since $\dot z^n\equiv0$ we have
\[\begin{split}
&|\int_{I_n}\a(s; \phih{s}{t}y, \phih{s}{t}z)\d t-\int_{I_n}\a(t;y,z)\d t|=\dots\\
&=\Big|\int_{I_n}\mint{\Gamma^h(s)}{\sgradh{s}\phih{s}{t}y\left(\DS(s,s)-\DS(s,t)\right)\sgradh{s}\phih{s}{t}z+\lambda \phih{s}{t}y(J_{s,h}^s-J_{t,h}^s)\phih{s}{t}z}{\Gamma^h(s)\d t}\Big|\,.
\end{split}\]
The lemma follows from the fact that $\Phih{s}{t}$ it linear on each $T_h^i(s)$ and  globally Lipschitz in time, as by Lemma \ref{L:Phih}.
\end{proof}

Let us formulate a crucial stability assertion for the adjoint scheme \eqref{S:DiscrAdj}.
\begin{Lemma}\label{L:AdjStab}
Let $z\in \Vk$ solve \eqref{S:DiscrAdj} with right-hand side $g\in L^2_{L^2(\Gamma^h)}$. For sufficiently small $k>0$ there exists $C>0$, depending only on $\Gamma$, such that
\[
\max_{1\le n\le N}\a(t_n;z^n,z^n)+ \frac{1}{k}\sum_{n=1}^N\|z^{n+1}-z^n\|^2_n+k\sum_{n=1}^N \|z^n\|^2_{H^1(\Gamma(t_n))}\le C\|g\|^2_{h,k}\,.
\]
\end{Lemma}
\begin{proof}
Apply \eqref{S:DiscrAdj} to $z^n$ to obtain
\[
\langle z^n-z^{n+1},z^n\rangle_{n} + k \a_n(z^n,z^n)=\mint{I_n}{\langle \phih{t_n}{t} g,z^n\rangle_{n}}{t}\,.
\]
This leads to
\[
\begin{split}
\frac{1}{2}\left(\|z^n\|^2_n+\|z^{n+1}-z^n\|^2_n-\|z^{n+1}\|^2_n\right)+ k \a_n(z^n,z^n)=\mint{I_n}{\langle \phih{t_n}{t} g,z^n\rangle_{n}}{t}&\\
\le \mint{I_n}{\| \phih{t_n}{t}  g\|_{L^2(\Gamma^h(t_n))}}{t}\|z^n\|_n
\le  \frac{1}{2k\M}\left(\mint{I_n}{\| \phih{t_n}{t}  g\|_{n}}{t}\right)^2+\frac{k\M}{2}\|z^n\|_n^2&\,.
\end{split}
\]
 Summing up and using \eqref{E:L2NormVariation} gives us
\[
\sum_{n=1}^N \left(\frac{1}{2}\|z^{n+1}-z^n\|^2_n-\M k\left(1+\frac{1}{2}C_{L^2(\Gamma^h)}\M k \right)\|z^n\|^2_{n} + k \a_n(z^n,z^n)\right)
\le \frac{1}{2\M}\| g\|^2_{h,k}\,,
\]
such that for $0<k<\min\left(\frac{2}{C_{L^2(\Gamma^h)}\M^2 },\frac{1}{2C_\r}\right)$
\begin{equation}\label{E:H1Stability}\begin{split}
\frac{k}{2}\sum_{n=1}^N \|z^n\|^2_{H^1(\Gamma(t_n))}&\le k\sum_{n=1}^N \left( \a(t_n;z^n,z^n)+\r_n(z^n,z^n)-\left(1+\frac{C_{L^2(\Gamma^h)}\M k}{2}\right){\M }\|z^n\|_n^2\right)\\
&\le \frac{1}{2\M}\| g\|^2_{h,k}\,.
\end{split}\end{equation}
Now we test \eqref{S:DiscrAdj} with $z^n-z^{n+1}$ to get
\[\begin{split}
\|z^n-z^{n+1}\|^2_n+\frac{k}{2}\left(\a_n(z^n,z^n)+\a_n(z^{n+1}-z^n,z^{n+1}-z^n)-\a_n(z^{n+1},z^{n+1})\right)=\dots\\
=\mint{I_n}{\langle  \phih{t_n}{t} g,z^n-z^{n+1}\rangle_n}{t}\le \frac{1}{2}\left(\mint{I_n}{\| \phih{t_n}{t} g\|^2_{n}}{t}\right)^2+\frac{1}{2} \|z^n-z^{n+1}\|^2_n\,.
\end{split}\]
Summing up and using Lemma \ref{L:H1cont} on $\a$ as well as the estimate  \eqref{E:r} on $\r$ we arrive at
\[\begin{split}
\frac{k}{2}&\a(t_m,z^m,z^m)+\frac{1}{2}\sum_{n=m}^N\Bigg(\|z^{n+1}-z^n\|^2_n\Bigg)\\ 
&\le \frac{1}{2}k\|g\|^2_{h,k}+ \frac{k}{2}\sum_{n=m+1}^N  \a(t_{n-1};z^{n},z^{n})-\a(t_n;z^{n},z^{n})+\r_{n-1}(z^{n},z^{n})-\r_{n}(z^{n},z^{n})\\
&\le \frac{1}{2}k\|g\|^2_{h,k}+ \frac{k}{2}\sum_{n=m+1}^N   C k\left(\|z^n\|^2_{H^1(\Gamma(t_n))}+\|z^n\|^2_{H^1(\Gamma(t_{n-1}))}\right)\,.
\end{split}\]
Combine with  \eqref{E:H1Stability}  to arrive at the lemma.
\end{proof}

The following Lemma shows, that it is sufficient to estimate the approximation error at the points $t_n$, $1\le n \le N$ to prove convergence in $L^2_{L^2(\Gamma^h)}$.
\begin{Lemma}\label{L:DerivativeBound}
 Let $r\in H^1([0,T],V)$, $V$ a separable Hilbert space, then there holds for $\tau \in I_n$
\[
\|r-r(\tau)\|_{L^2(I_n,V)}\le k\|r'\|_{L^2(I_n,V)}\,.
\]
In our situation this implies for $r\in H^1_{V_h}$ that
\begin{enumerate}
\item  $k\|r(\tau)-\bar r^n\|^2_{L^2(\Gamma^h(\tau))}\le C k^2 \mint{I_n}{\|\dot r\|^2_{L^2(\Gamma^h(t))}}{t}\,,$
\item and $\mint{I_n}{\|r(t)-\bar r^n\|^2_{L^2(\Gamma^h(t))}}{t}\le C k^2 \mint{I_n}{\|\dot r\|^2_{L^2(\Gamma^h(t))}}{t}\,.$
\end{enumerate}
\end{Lemma}
\begin{proof}
For the fist assertion approximate $r$ by $r_i\in\mathcal D([0,T],V)$ such that $r_i\stackrel{H^1([0,T],V)}{\longrightarrow}r$ as $i\rightarrow\infty$. Use
\[
\|r_i-r_i(\tau)\|_{L^2(I_n,V)}=\left(\mint{I_n}{\left\|\int_\tau^t r_i'(\theta)\d \theta \right\|_V^2}{t}\right)^\frac{1}{2}\le \left(\mint{I_n}{k \int_{I_n} \left\| r_i'(\theta) \right\|_V^2\d \theta}{t}\right)^\frac{1}{2}\le k\|r_i'\|_{L^2(I_n,V)}\,,
\]
and the fact that $r\in C([0,T],V)$, compare \cite[Thm. 3.1]{LionsMagenes1968}. Hence the first part of the lemma follows by passing to the limit.

In our situation this implies, since $\phih{\tau}{t} r(t) \in H^1([0,T],V_h(\tau))$
\[\begin{split}
\|\bar r^n-r(\tau)\|_{L^2(\Gamma^h(\tau))}^2&=\Big \| \frac{1}{k}\mint{I_n}{\phih{\tau}{t} r(t)-r(\tau)}{t}\Big\|^2_{L^2(\Gamma^h(\tau))} \le \frac{1}{k}\mint{I_n}{  \Big \|\phih{\tau}{t} r(t)-r(\tau)\Big\|^2_{L^2(\Gamma^h(\tau))} }{t} \\
&\le  k  \mint{I_n}{\|\big(\phih{\tau}{t} r(t)\big)' \|^2_{L^2(\Gamma^h(\tau))}}{t}
 \le kC_J^h \mint{I_n}{\|\dot r\|^2_{L^2(\Gamma^h(t))}}{t}\,.
\end{split}\]
This proves  \emph{1.}, in order to get \emph{2.} integrate over $I_n$.
\end{proof}
We are now prepared to prove the main result of this section.

\begin{Theorem}\label{L:TimeConv}
Let $f\in L^2_{L^2(\Gamma)}$, and let $y_h$ and $y_{h,k}$ solve \eqref{E:WeakTransport_d} and \eqref{E:FullyDisc}, respectively, with $y_0^h\in L^2(\Gamma_0^h)$  and $f_h=f_l$. There exists a constant $C>0$ independent of $h,k>0$ and of $f$ and $y_0^h$ such that
\[
\|y_h-y_{h,k}\|_{L^2_{L^2(\Gamma^h)}}\le Ck\left(\|\dot y_h\|_{L^2_{L^2(\Gamma^h)}}+ \|f\|_{L^2_{L^2(\Gamma)}}+\|y_0^h\|_{L^2(\Gamma_0^h)}\right)\,.
\]
\end{Theorem}
\begin{proof}
The proof is inspired by \cite[Thm. 5.2]{SinhaDeka2005}, compare also \cite[Thm 1.2.5]{VierlingDiplomaThesis} and \cite[Thm 5.1]{MeidnerVexler2008_I}.
Test \eqref{E:WeakTransport_d} with $\phih{t}{t_n}\varphi$, $\varphi\in V_h$ and integrate over $I_n$ to obtain
\begin{equation}\label{E:IntWeak}
\langle{y_h(t_n),\varphi}\rangle_n-\langle{y_h(t_{n-1}),\varphi}\rangle_{n-1}+\mint{[0,T]}{\a(t;y_h,\varphi)}{t}=\mint{[0,T]}{ \langle f_l,\varphi\rangle_{L^2({\Gamma^h(t)})}}{t}\,.
\end{equation}
Solve the adjoint equation \eqref{S:DiscrAdj} for $z$ with both right-hand side and test function $\varphi=g=\sum_{n=1}^N (\bar y_h^n-y_{h,k}^n)\mathbf 1_{I_n}$
\begin{equation}\label{E:AdjHelp}
\mint{I_n}{\|\bar y_h^n-y_{h,k}^n\|^2_n}{t}=\langle z^n-z^{n+1},\bar y_h^n-y_{h,k}^n\rangle_{n} + k \a_n(\bar y_h^n-y_{h,k}^n,z^n)
\end{equation}
Subtract \eqref{E:IntWeak} from \eqref{E:FullyDisc}. Tested with $z$ this yields
\[\begin{split}
\langle{y_{h,k}^n-y_h(t_n),z^n}\rangle_n-\langle{y_{h,k}^{n-1}-y_h(t_{n-1}),z^n}\rangle_{n-1}+k{\a_n(y_{h,k}^n-\bar y_h^n,z^n)}=\dots\\
=\mint{I_n}{\a(t;y_h,z^n)}{t}-k{\a_n(\bar y_h^n,z^n)}+k\langle \bar f_l^n,z^n\rangle_n-\mint{I_n}{\langle f_l,z^n\rangle_{L^2(\Gamma^h(t))}}{t}
\end{split}\]
Let $\bar y_h=\sum_{n=1}^N \bar y_h^n\mathbf 1_{I_n}$. Add  \eqref{E:AdjHelp} and sum up over $1\le n\le N$ to get 
\[\begin{split}
\langle& f_l,z\rangle_{h,k}-\langle f_l,z\rangle_{L^2_{L^2(\Gamma^h)}}+\sum_{n=1}^N\mint{I_n}{\|\bar y_h-y_{h,k}\|^2_n}{t}+\mint{I_n}{\a(t;y_h,z^n)}{t}-k{\a(t_n;\bar y_h^n,z^n)}=\dots\\
&=\sum_{n=1}^N\left[k\r_n(\bar y_h^n,z^n)+\langle{\bar y_{h}^n-y_h(t_n),z^n}\rangle_n-\langle{y_{h,k}^{n-1}-y_h(t_{n-1}),z^n}\rangle_{n-1}-\langle z^{n+1},\bar y_h^n-y_{h,k}^n\rangle_{n}\right]\\
&=\langle{y_{h,k}^{N}-y_h(t_{N}),z^{N+1}}\rangle_{N}-\langle{y_{h,k}^{0}-y_h(t_{0}),z^1}\rangle_{0}+\sum_{n=1}^Nk\r_n(\bar y_h^n,z^n)+\langle{\bar y_{h}^n-y_h(t_n),z^n-z^{n+1}}\rangle_n\\
&=\sum_{n=1}^Nk\r_n(\bar y_h^n,z^n)+\langle{\bar y_{h}^n-y_h(t_n),z^n-z^{n+1}}\rangle_n\,,
\end{split}\]
and finally, bringing to bear everything we have, i.e. the estimates from  Lemma \ref{L:H1cont} for $\a$, from Lemma \ref{L:ScalarProduct} for the $L^2$-norms, and the bound on $\r$ from \eqref{E:r}, we arrive at
\[\begin{split}
\sum_{n=1}^N&\mint{I_n}{\|\bar y_h-y_{h,k}\|^2_n}{t}\le \left(k\sum_{n=1}^N\|\bar y_{h}^n-y_h(t_n)\|_n^2\right)^\frac{1}{2}\left(\frac{1}{k}\sum_{n=1}^N\|z^n-z^{n+1}\|_n^2\right)^\frac{1}{2} +\dots\\
+&C\left(k\sum_{n=1}^N\left(\mint{I_n}{\|\phih{t_n}{t}y_h\|_{H^1(\Gamma^h(t_n))}}{t}\right)^2\right)^\frac{1}{2}\left(k\sum_{n=1}^N\|z^n\|_{H^1(\Gamma^h(t_n))}^2 \right)^\frac{1}{2}+ Ck\|f\|_{L^2_{L^2(\Gamma)}}\underbrace{\|z^l\|_{L^2_{L^2(\Gamma)}}}_{\le C\|z\|_{h,k}}\,.
\end{split}\]
Hence using Lemma \ref{L:AdjStab} on $z$ we can divide by $\|\bar y_h-y_{h,k}\|_{h,k}$.  The Lemmas \ref{L:ScalarProduct} and \ref{L:H1cont} allow us to estimate the  involved norms, and because of the stability of the space discretization, compare Lemma \ref{L:Stability}, we can estimate the $H^1(\Gamma^h(t))$-term,  to finally arrive at
\begin{equation}\label{E:MainEst}
\|\bar y_h-y_{h,k}\|_{L^2_{L^2(\Gamma^h)}}\le C\left(\left(k\sum_{n=1}^N\|\bar y_{h}^n-y_h(t_n)\|_n^2\right)^\frac{1}{2}+ k\|f\|_{L^2_{L^2(\Gamma)}}+k\|y_0^h\|_{L^2(\Gamma^h_0)}\right)\,.
\end{equation}
We now apply Lemma \ref{L:DerivativeBound}[\emph{2.}] to the error $e_k=y_{h,k}-y_h$ and the averaged error $\bar e_k=y_{h,k}-\bar y_h$ and sum up to obtain $\|e_k-\bar e_k\|_{L^2_{L^2(\Gamma^h)}}\le Ck\|\dot y_h\|_{L^2_{L^2(\Gamma^h)}}$. Combine with  \eqref{E:MainEst} and \ref{L:DerivativeBound}[\emph{1.}] to estimate
\[
\|e_k\|_{L^2_{L^2(\Gamma^h)}}\le C k\|\dot y_h\|_{L^2_{L^2(\Gamma^h)}}+  \|\bar e_k\|_{L^2_{L^2(\Gamma^h)}}\le Ck\left(\|\dot y_h\|_{L^2_{L^2(\Gamma^h)}}+ \|f\|_{L^2_{L^2(\Gamma)}}+\|y_0^h\|_{L^2(\Gamma_0^h)}\right)\,.
\]
\end{proof}
With view of the stability assertions from \eqref{E:Smoothness} and Lemma \ref{L:Stability} and together with Theorem \ref{L:SpaceConv} we get the following Corollary.
\begin{Corollary} \label{C:StrongOrderConv} In the situation of Theorem \ref{L:TimeConv} let in addition $\lambda=\mu_l$ and  $y_0\in H^2(\Gamma_0)$, and choose $y_0^h$ as the piecewise linear interpolation of $(y_0)_l$. There exists a constant $C>0$ independent of $h,k>0$ and of $f$ and $y_0$  such that
$$\|y_{h,k}^l-y\|_{L^2_{L^2(\Gamma)}}\le C(h^2+k)\left(\|y_0\|_{H^2(\Gamma_0)}+ \| f \|_{L^2_{L^2(\Gamma)}}\right)\,.$$
\end{Corollary}
As addressed in Remark \ref{R:Projection}, it should be possible to relax the condition on $y_0$ into $y_0\in H^1(\Gamma_0)$ using the $L^2(\Gamma_0)$-projection or the $L^2(\Gamma_0^h)$-projection $P_0^h$.

But even in the case of low regularity we still get a uniform estimate.
\begin{Corollary} \label{C:WeakOrderConv} In the situation of Theorem \ref{L:TimeConv} let only $y_0\in L^2(\Gamma_0)$ hold while $f\equiv 0$. Let further $y_0^h=P_0^h((y_0)_l)$. There exists a constant $C>0$ independent of $h,k>0$ and of $y_0$  such that
$$\|y_{h,k}^l-y\|_{L^2_{L^2(\Gamma)}}\le C\left(h+\sup_{t\in[0,T]}\|\lambda^l-\mu\|_{L^\infty(\Gamma(t))}+\frac{k}{h}\right)\|y_0\|_{L^2(\Gamma_0)}\,.$$
\end{Corollary}
\begin{proof}
Regarding Theorem \ref{L:WeakSpaceConv} and \ref{L:TimeConv} it remains to bound $\|\dot y_h\|_{L^2_{L^2(\Gamma^h)}}$. Like in the proof of Theorem \ref{L:WeakSpaceConv}, using Lemma \ref{L:Stability} and an inverse estimate, we arrive at the desired estimate.
\end{proof}
In particular, for $\kappa>0$, choose $k= \kappa h^2$ and $\lambda$ such that $\sup_{t\in[0,T]}\|\lambda^l-\mu\|_{L^\infty(\Gamma(t))}\le Ch$ to get an $\mathcal O(h)$-convergent scheme.

\begin{Remark}\label{R:lambda}Note that our freedom in the choice of $\r$ now allows us to finally drop the conditions on $\lambda$ and $\mu$, respectively, in \eqref{E:WeakTransport_lambda}, \eqref{E:WeakTransport_d}, and \eqref{E:big_lambda}. Let us assume we want to approximate the solution $y$ of \eqref{E:WeakTransport_lambda} with $\mu\equiv 0$, $y_0\in H^1(\Gamma(0))$, and $f\in L^2_{L^2(\Gamma)}$. Now $y_{h,k}\in\Vk$ solves
\[
\begin{split}
y_{h,k}^0=y_0^h\,,\quad&\forall \varphi\in \Vk\,,\;1\le n\le N\,:\\
\langle y_{h,k}^n,&\varphi\rangle_n -\langle y_{h,k}^{n-1},\varphi\rangle_{n-1} + k \mint{\Gamma^h(t_n)}{\nabla_{\Gamma^h(t_n)} y_{h,k}^n\cdot\nabla_{\Gamma^h(t_n)}\varphi}{\Gamma^h(t_n)}=k\langle\bar f_h^n,\varphi\rangle_{n}\,,
\end{split}
\]
iff $y_{h,m,\lambda}=\sum_{n=1}^N e^{-\lambda t_n}y_{h,k}^n\mathbf 1_{I_n}\in\Vk$, $\lambda>0$ solves 
\[
\begin{split}
y_{h,k,\lambda}^0=y_0^h\,,\quad\forall \varphi\in \Vk\,,&\;1\le n\le N\,:\\
\langle y_{h,k\lambda}^n,\varphi\rangle_n -\langle y_{h,k,\lambda}^{n-1},&\varphi\rangle_{n-1} + k \mint{\Gamma^h(t_n)}{\nabla_{\Gamma^h(t_n)} y_{h,k}^n\cdot\nabla_{\Gamma^h(t_n)}\varphi+\lambda y_{h,k\lambda}^n\varphi}{\Gamma^h(t_n)}+k\r_n(y^n,\varphi)\\
&=k\langle e^{-\lambda t_{n-1}}\bar f_h^n,\varphi\rangle_{n}\,,
\end{split}
\]
with $$k\r_n(\psi,\varphi)= (e^{\lambda k}-1-\lambda k)\langle \psi,\varphi\rangle_n+k (e^{\lambda k}-1)\mint{\Gamma^h(t_n)}{\nabla_{\Gamma^h(t_n)} \psi\cdot\nabla_{\Gamma^h(t_n)}\varphi}{\Gamma^h(t_n)}\,.$$ Taking into account that $\|e^{-\lambda t} f(t)-\sum_{n=1}^Ne^{-\lambda t_n}\mathbf 1_{I_n}f(t)\|_{L^2_{L^2(\Gamma)}}\le  k\|f\|_{L^2_{L^2(\Gamma)}}$, we apply Corollary  \ref{C:StrongOrderConv} to $y_{h,m,\lambda}$ and conclude $\|y_{h,k}^l-y\|_{L^2_{L^2(\Gamma)}}\le Ce^{\lambda T}(h^2+k)$.
\end{Remark}

\section{Variational Discretization}\label{S:VarDisc}

We now return to problem $(\mathbb P_d)$ which has the advantage over $(\mathbb P_T)$, that its adjoint equation satisfies the regularity estimate \eqref{E:Smoothness}. For $(\mathbb P_T)$ this is not  the case iff  $y_T\in L^2(\Gamma(T))\setminus H^1(\Gamma(T))$. In the spirit of \cite{Hinze2005}, let us approximate $(\mathbb P_d)$ by 
\begin{equation*}
(\mathbb{P}_d^h)\quad \left\{\begin{array}{l}
\min_{u \in L^2_{L^2(\Gamma^h)}} J(u) := \frac{1}{2}\|S_d^h(u)-(y_d)_l\|_{h,k}^2 + \frac{\alpha}{2} \|u\|_{h,k }\\
\mbox{s.t. }a\le u\le b\,,
\end{array}\right.
\end{equation*}
with $\{\Gamma^h(t)\}_{t\in[0,T]}$ as in Section \ref{S:FEDisc} and $S_d^h:(L^2_{L^2(\Gamma^h)},\langle\cdot,\cdot\rangle_{h,k})\rightarrow (L^2_{L^2(\Gamma^h)},\langle\cdot,\cdot\rangle_{h,k})$, $f_h\mapsto y_f$ is defined through the scheme \ref{E:FullyDisc} with $\lambda\equiv 0$ and $y_0^h\equiv 0$. We choose the scalar product $\langle\cdot,\cdot\rangle_{h,k}$ defined in \eqref{E:DiscSP} in order to obtain a computable scheme to evaluate ${S_d^h}^*$, namely \eqref{S:DiscrAdj} with $z^{N+1}=0$. Given snapshots $\{\Gamma^h(t_n)\}_{n=1}^N$, the product $\langle\cdot,\cdot\rangle_{h,k}$ can be evaluated exactly for functions  $\varphi_h\in\Vk$ as well as  for $P_{[a,b]}(\varphi_h)$.

Let $U_{\textup{ad}}^h=\set{ v\in L^2_{L^2(\Gamma^h)}}{a\le v\le b}$. As in \eqref{E:NecCond} the first order necessary optimality condition for an optimum $u_h$ of $(\mathbb{P}_d^h)$ is
\begin{equation}\label{E:NecCondDisc}
\langle \alpha u_h + {S_d^h}^*(S_d^hu_h -(y_d)_l),v-u_h\rangle_{h,k}\ge 0\,,\quad \forall v\in U_\textup{ad}\,.
\end{equation}
First note that as in the continuous case the $\langle\cdot,\cdot\rangle_{h,k}$-orthogonal projection onto $U_{\textup{ad}}^h$ coincides with the point-wise projection $P_{[a,b]}(v)$. 
Similar to \ref{E:FONC} we get 
\begin{equation}\label{E:FONCd}
u_h = P_{[a,b]}\left(-\frac{1}{\alpha}p_d^h(u)\right)\,,\quad p_d^h(u)={S_d^h}^*\left(S_d^hu-(y_d)_l\right)\,.
\end{equation}
Equation \eqref{E:FONCd} is amenable to a semi-smooth Newton method that, while still being implementable, operates entirely in $L^2_{L^2(\Gamma^h)}$. The implementation requires one  to resolve the boundary between the inactive set $\mathcal I_u(t_n)=\set{\gamma\in \Gamma(t_n)}{a< -\frac{1}{\alpha}p_d^h(u)[\gamma]<b}$ and the active set $\mathcal A_u(t_n) =\Gamma^h(t_n)\setminus \mathcal I_u(t_n)$ for $1\le n\le N$. For details on the implementation see \cite{HinzeVierling2010_Surf} and \cite{HinzeVierling2010}.
Note that in order to implement $S_d^h$ and ${S_d^h}^*$  according to \eqref{E:FullyDisc} and \eqref{S:DiscrAdj} for right-hand sides in $\Vk$, again one only needs to know the snapshots $\{\Gamma^h(t_n)\}_{n=0}^N$.
The solution of $(\mathbb P_d^h)$ converges towards that of $(\mathbb P_d)$ and the order of convergence is optimal in the sense that it is given by the order of convergence of $S_d^h$ and ${S_d^h}^*$.
\begin{Theorem}[Order of Convergence for $(\mathbb P_d^h)$]\label{T:Convergence}
Let $u \in L^2_{L^2(\Gamma)}$, $u_h\in L^2_{L^2(\Gamma^h)}$ be the solutions of $(\mathbb P_d)$ and $(\mathbb P_d^h)$, respectively. Let $C>1$. Then for sufficiently small $ h,k>0$ there holds
\begin{equation*}
\begin{split}
2\alpha \big\|u^l_h-u\big\|^2_{L^2_{L^2(\Gamma)}}+\big\|y^l_h-y\big\|_{L^2_{L^2(\Gamma)}}^2\le C\bigg(&2\big \langle \left(\lup{S_d^h}^*\ldown - S_d^*\right)(y-y_d),u-u_h^l\big\rangle_{L^2_{L^2(\Gamma)}}\dots\\
&+\left\|\left(\lup S_d^h\ldown-S_d\right)u\right\|_{L^2_{L^2(\Gamma)}}^2\bigg)\,,
\end{split}\end{equation*}
with $y=S_du$ and $y_h=S_d^hu_h$.
\end{Theorem}
\begin{proof} The proof is a modification of the one from \cite[Thm. 3.4]{HinzePinnauUlbrich2009}, compare also \cite{HinzeVierling2010_Surf}.
Let $ \pr{U_\textup{ad}^h}{\cdot}$ denote the $\langle\cdot,\cdot\rangle_{h,k}$-orthogonal projection onto $U_{\textup{ad}}^h$. We have
$$
u_l = P_{[a,b]}\left(-\frac{1}{\alpha}p_d(u)\right)_l = P_{[a,b]}\left(-\frac{1}{\alpha}p_d(u)_l\right)=\pr{U_\textup{ad}^h}{-\frac{1}{\alpha}p_d(u)_l}\,.
$$
Since $u_h\in U_\textup{ad}^h$, from the characterization of  $\pr{U_{ad}^h}{\cdot}$ it follows
$$
\langle-\frac{1}{\alpha}p_d(u)_l-u_l,u_h-u_l\rangle_{h,k}\le 0\,.
$$
On the other hand we can plug $u_l$ into  \eqref{E:NecCondDisc} and get
$$
\langle \alpha u_h + p_d^h( u_h),u_l-u_h\rangle_{h,k}\ge 0\,.
$$
Adding these inequalities yields
\begin{equation*}
\begin{split}
\alpha\|u_l- u_h\|^2_{h,k}\le&\big\langle\left(p_d^h( u_h)-p_d(u)_l\right),u_l- u_h\big\rangle_{h,k}\\
=&\langle p_d^h(u_h) -{S_d^h}^*(y-y_d)_l,u_l-u_h\rangle_{h,k} + \langle {S_d^h}^*(y-y_d)_l-p_d(u)_l,u_l-u_h\rangle_{h,k}\,.
\end{split}
\end{equation*}
The first addend is estimated via
\[\begin{split}
\langle p_d^h(u_h) -(S_d^h)^*(y-y_d)_l,u_l-u_h\rangle_{h,k} &= \langle y_h-y_l,S_d^hu_l-y_h\rangle_{h,k}\\
&=-\|y_h-y_l\|^2_{h,k} +\langle y_h-y_l,S_d^hu_l-y_l\rangle_{h,k}\\
&\le - \frac{1}{2}\|y_h-y_l\|^2_{h,k}+ \frac{1}{2}\|S_d^hu_l-y_l\|^2_{h,k}\,.
\end{split}\]
This yields
\[
2\alpha\|u_l- u_h\|^2_{h,k} + \|y_h-y_l\|^2_{h,k}\le 2 \langle ({S_d^h}^*\ldown - \ldown S_d^*)(y-z),u_l-u_h\rangle_{h,k} + \|S_d^hu_l-y_l\|^2_{h,k}\,.
\]
The claim follows for sufficiently small $h,k>0$, using the equivalence of the involved norms stated in  Lemma \ref{L:ScalarProduct}. 
\end{proof}
For the problem
\begin{equation*}
(\mathbb{P}_T^h)\quad \left\{\begin{array}{l}
\min_{u\in L^2_{L^2(\Gamma^h)}} J(u) := \frac{1}{2}\|S_T^h(u)-(y_T)_l\|_{L^2(\Gamma^h(T))}^2 + \frac{\alpha}{2} \|u\|_{L^2_{L^2(\Gamma^h)} }\\
\mbox{s.t. }a\le u\le b\,,
\end{array}\right.
\end{equation*}
one can prove a similar result. Here the operator $S_T^h$ is the map $f_h\rightarrow y_f(T)$, according to the scheme \eqref{E:FullyDisc} with $\lambda\equiv 0$.
\begin{Theorem}[Order of Convergence for $(\mathbb P_T^h)$]\label{T:Convergence_T}
Let $u \in L^2_{L^2(\Gamma)}$, $u_h\in L^2_{L^2(\Gamma^h)}$ be the solutions of $(\mathbb P_T)$ and $(\mathbb P_T^h)$, respectively. Let $C>1$. Then for sufficiently small $ h,k>0$ there holds
\begin{equation*}
\begin{split}
2\alpha \big\|u^l_h-u\big\|^2_{L^2_{L^2(\Gamma)}}+\big\|y^l_h-y\big\|_{{L^2(\Gamma(T))}}^2\le C\bigg(&2\big \langle \left(\lup{S_T^h}^*\ldown - S_T^*\right)(y-y_T),u-u_h^l\big\rangle_{L^2_{L^2(\Gamma)}}\dots\\
&+\left\|\left(\lup S_T^h\ldown-S_T\right)u\right\|_{{L^2(\Gamma(T))}}^2\bigg)\,,
\end{split}\end{equation*}
with $y=S_T u$ and $y_h=S_T^hu_h$.
\end{Theorem}

Now as to the convergence of $\left(\lup {S_d^h}^*\ldown- S_d^*\right)$, note that  taking the adjoint does not commute with the discretization. Indeed, apply the scheme \eqref{E:FullyDisc} to the adjoint equation \eqref{E:AdjSolOp}, i.e. $\lambda =- (\div_{\Gamma(t_n)} V)_l$ to get
\[
\begin{split}
&z_g^{N+1}=0\,,\quad\forall \varphi\in \Vk\,,\;1\le n\le N\,:\\
&\mint{I_n}{\langle \phih{t_n}{t}g_h,\varphi\rangle_{n}}{t}=\langle z_g^n,\varphi\rangle_n -\langle z_g^{n+1},\varphi\rangle_{n} +\dots\\
&k\mint{\Gamma^h(t_n)}{\nabla_{\Gamma^h(t_n)}\varphi\nabla_{\Gamma^h(t_n)}z_g^n-(\div_{\Gamma(t_n)} V)_l\varphi z^n}{\Gamma^h(t_n)}+\mint{I_n}{\langle\varphi\div_{\Gamma^h(t)} V_h, z^n\rangle_{L^2(\Gamma^h(t))}}{t}\,,
\end{split}
\]
instead of \eqref{S:DiscrAdj}.

In the situation of $(\mathbb P_d^h)$ however, this discrepancy can  be remedied by Lemma \ref{L:LiftContinuity} which implies 
\begin{equation*}
\|\ldown-\lup^*\|_{\mathcal L(L^2_{L^2(\Gamma)},L^2_{L^2(\Gamma^h)})},\|\lup-\ldown^*\|_{\mathcal L(L^2_{L^2(\Gamma^h)},L^2_{L^2(\Gamma)})}\le C h^2\,,
\end{equation*}
and due to Lemma \ref{L:ScalarProduct} which allows us to conclude
\begin{equation}\label{E:LiftAdjConv}
\|\ldown-\lup^*\|_{\mathcal L(L^2_{L^2(\Gamma)}, (L^2_{L^2(\Gamma^h)},\langle\cdot,\cdot\rangle_{h,k}))},\|\lup-\ldown^*\|_{\mathcal L(L^2_{L^2(\Gamma^h)}, (L^2_{L^2(\Gamma^h)},\langle\cdot,\cdot\rangle_{h,k}))}\le C (h^2+k)\,,
\end{equation}
if we interpret $\ldown,\lup$ as operators into or on $ (L^2_{L^2(\Gamma^h)},\langle\cdot,\cdot\rangle_{h,k})$, respectively.

Hence we get the estimate
\[\begin{split}
\left\|\lup {S_d^h}^*\ldown- S_d^*\right\|&\le \left\|(\lup -\ldown^*){S_d^h}^*\ldown \right\|+\left\|\ldown^*{S_d^h}^*(\ldown-\lup^*) \right\|+\left\|\ldown^* {S_d^h}^*\lup^*-  S_d^*\right\|\\
&\le C(k+h^2)\,,
\end{split}\]
in the $\mathcal L(L^2_{L^2(\Gamma)},L^2_{L^2(\Gamma)})$-operator norm.

Now as to $(\mathbb P_T)$, all the results from section \ref{S:FEDisc} and \ref{S:DGDisc} remain valid under the time transform $t'=T-t$. As opposed to problem $(\mathbb P_d^h)$, here it is easier to proof the convergence of $S_T^{h*}$ than that of $S_T^h$ itself. In order to discretize ${S_T}^*$ we choose $\lambda = -\div_{\Gamma^h(t)}V_h$ to approximate $\mu_l=-(\div_{\Gamma(t)}V)_l$ and 
$$\r_n(\psi,\varphi) = \mint{I_n}{\langle\varphi\div_{\Gamma^h(t)} V_h, z^n\rangle_{L^2(\Gamma^h(t))}}{t}-k\langle{\varphi\div_{\Gamma^h(t_n)} V_h, z^n}\rangle_{L^2(\Gamma^h(t_n))}\,,$$
and apply Theorem \ref{C:WeakOrderConv} to end up with $\|\lup{S_T^h}^*\ldown-S_T^*\|_{\mathcal L(L^2(\Gamma(T)),L^2_{L^2(\Gamma)})}\le C (h+\frac{k}{h})$, where ${S_T^h}^*:z_T\mapsto z\in \Vk\subset (L^2_{L^2(\Gamma^h)},\langle\cdot,\cdot\rangle_{h,k})$ according to
\[
\begin{split}
z^{N+1}=z_T\,,\quad&\forall \varphi\in \Vk\,,\;1\le n\le N\,:\\
&\langle z^n,\varphi\rangle_n -\langle z^{n+1},\varphi\rangle_{n} + k \mint{\Gamma^h(t_n)}{\nabla_{\Gamma^h(t_n)} z^n\cdot\nabla_{\Gamma^h(t_n)}\varphi}{\Gamma^h(t_n)}=0\,.
\end{split}
\]
Now in addition to \eqref{E:LiftAdjConv} we have 
\[
\|\ldown-\lup^*\|_{\mathcal L(L^2(\Gamma(T)),L^2(\Gamma^h(T)))}\le Ch^2\,,
\]
due to the inclusion \eqref{E:IntBnd}.  We conclude
\[
\left\|\lup {S_T^h}\ldown- S_T\right\|_{\mathcal L(L^2_{L^2(\Gamma)},L^2(\Gamma(T)))}\le C (h+\frac{k}{h})\,,
\]
the operator $S_T^h={S_T^h}^{**}:(L^2_{L^2(\Gamma^h)},\langle\cdot,\cdot\rangle_{h,k})\rightarrow L^2(\Gamma^h(T)))$, $f_h\mapsto y(T)$ being defined by the scheme
\[
\begin{split}
y^0\equiv 0\,,\quad&\forall \varphi\in \Vk\,,\;1\le n\le N\,:\\
&\langle y^n,\varphi\rangle_n -\langle y^{n-1},\varphi\rangle_{n-1} + k \mint{\Gamma^h(t_n)}{\nabla_{\Gamma^h(t_n)} y^n\cdot\nabla_{\Gamma^h(t_n)}\varphi}{\Gamma^h(t_n)}=k\langle \bar f_h^n,\varphi\rangle_{n}\,,
\end{split}
\]
as shows summation over $n$.
If $y_T$ is more regular, such as $y_T\in H^1(\Gamma(T))$, then one might want to apply results from \cite{DziukElliott2011_PP} that state $h^2$-convergence of the discretization $S_T^h$, yet not in the $\mathcal L(L^2_{L^2(\Gamma)},L^2(\Gamma(T)))$-norm.  In order to to so, it remains to ensure the regularity assumptions of \cite[Thm. 4.4]{DziukElliott2011_PP} to be met by the optimal control $u$.

 \begin{figure}
\center
\includegraphics[width =\textwidth]{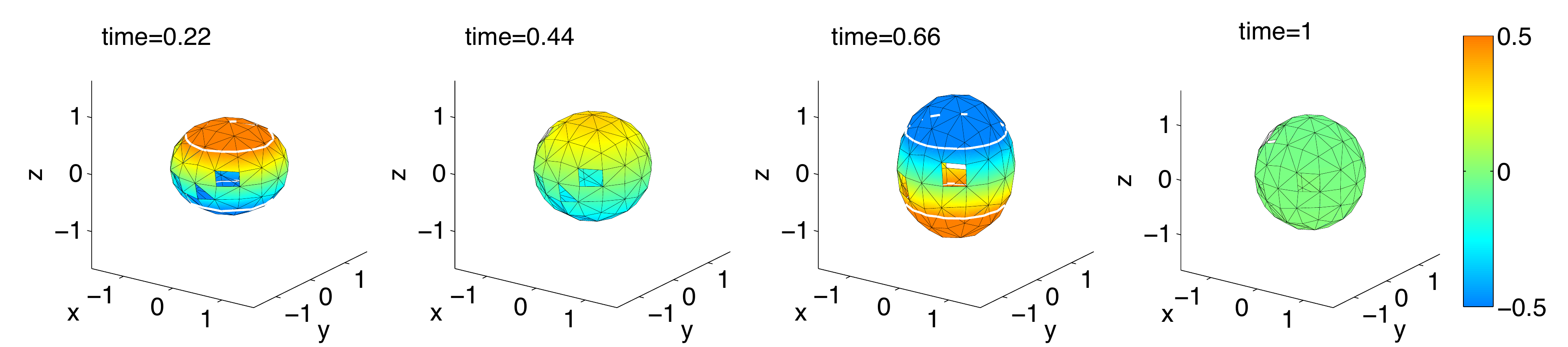}
\caption{Selected time snapshots of $\bar u_h$ computed for Example \ref{Ex:Sphere} on the Sphere after 4 refinements.}\label{Fig:Sphere}
\end{figure}

\section{Example}\label{S:Example}

Provided the results from \cite{HintermuellerItoKunisch2003} and \cite{Ulbrich2003} hold on surfaces, Equation \eqref{E:FONCd} is semi-smooth due to the smoothing properties of $S_d^{h*}$, i.e.  the stability ensured by Lemma \ref{L:AdjStab}. The lemma a priori holds only in the case $\lambda\ge 1$, but  can be extended for arbitrary $\lambda,\mu$ by rescaling, see Remark \ref{R:lambda}. 
By Lemma \ref{L:AdjStab} the operator $\phih{s}{\cdot}{S_d^h}^*$ continuously maps $(L^2_{L^2(\Gamma^h)},\langle\cdot,\cdot\rangle_{h,k})$ into $$L^\infty([0,T],H^1(\Gamma^h(s)))\subset L^p([0,T],L^p(\Gamma^h(s)))\simeq L^p([0,T]\times\Gamma^h(s))$$ for every $2<p<\infty$. This would imply semi-smoothness of the operator $$ P_{[a,b]}\left(-\frac{1}{\alpha}\phih{s}{t}\left(p_d^h\left(\phih{t}{s}(\cdot)\right)\right)\right):L^2([0,T]\times\Gamma^h(s))\rightarrow L^2([0,T]\times\Gamma^h(s))\,,$$ compare \cite{Ulbrich2003}, and thus of equation \eqref{E:FONCd}.

We implemented a semi-smooth Newton Algorithm for \eqref{E:FONCd}, along the lines of \cite{HinzeVierling2010}. 

\begin{Example}[High Regularity]\label{Ex:Sphere}
Consider problem $(\mathbb P_d)$ with $\alpha = 1$, $a=-\frac{1}{2}$, $b=\frac{1}{2}$, $T=1$, and $\Gamma_0\subset\R^3$ the unit sphere. Let $\Gamma(t)=\bar \Phi_0^t \Gamma_0$ with $\bar \Phi_0^t(x,y,z) = (x,y, {z}/{\rho^2(t)})^T$ and $\rho(t)=e^\frac{\sin(2\pi t)}{2}$. In coordinates $(x,y,z)$ of $\R^3$ let   $\bar u =P_{[-\frac{1}{2},\frac{1}{2}]}(z\sin(2\pi t))$ and $y_d = \tilde y_d + S_d\bar u$ with
   $$\tilde y_d = -\alpha\left(\left(\frac{\pi}{2}\sin(2\pi t)-2\pi\right)\cos(2\pi t) + \frac{\sin(2 \pi t)\rho(t)}{x^2+y^2+\rho^2 z^2}\left(\rho(t)+1 - z^2\frac{\rho^3(t)-\rho^2(t)}{x^2+y^2+\rho^2z^2}\right) \right)z\,.$$
   Then $\bar u$ solves $(\mathbb P_d)$.
\end{Example}

In order to compute the solution $\bar u_h$ of $(\mathbb P_d^h)$ we construct triangulations of $\Gamma_0$ from our macro-triangularion $R_0$, i.e. the cube whose nodes reside on $\Gamma_0$ triangulated into $12$ rectangular triangles. We generate $R_{i+1}$ from $R_i$  through  longest edge refinement followed by projecting the inserted vertices onto $\Gamma_0$.

Table \ref{T:Sphere} shows the relative error in the $L^2_{L^2(\Gamma^h)}$-norm and the relative $L^\infty$-error
$$ERR_\infty=\frac{\|\phih{s}{t}(\bar u_h-\bar u_l)\|_{L^\infty([0,T]\times \Gamma^h(s))}}{\|\phih{s}{t}\bar u_l\|_{L^\infty([0,T]\times \Gamma^h(s))}}\,,$$
as well as the corresponding experimental orders of convergence 
$$EOC_i = \ln\frac{ERR_i}{ERR_{i-q}}\ln \left(\ln\frac{H_i}{H_{i-q}}\right)^{-1}\,,$$
where $H$ denotes the maximal edge length of $\Gamma_0^h$, see Table \ref{T:LR}. Throughout this section we chose $q=2$ for  both $EOC_{L^2}$ and $EOC_{L^\infty}$, and the time step length is  $k=\frac{1}{20}H^2$.
\begin{table}
\begin{tabular}{r|cccc |r|cccc}
R &$ERR_{L^2}$  & $EOC_{L^2}$ & $ERR_{\infty}$ & $EOC_\infty$ &R&$ERR_{L^2}$  & $EOC_{L^2}$ & $ERR_{\infty}$ & $EOC_\infty$   \\
\hline
$0$ & 1.68e-01 &  -       & 8.71e-01 & -       & $5$ & 6.78e-03 & 2.15 & 1.08e-01 & 2.01 \\ 
$1$ & 5.40e-02 &  -       & 7.88e-01 & -       &  $6$ &3.15e-03 & 2.09 & 5.01e-02  & 1.97\\
$2$ & 4.13e-02 & 2.45 & 5.32e-01 & 0.86 & $7$ & 1.72e-03 & 2.03 & 2.80e-02 & 1.99\\
$3$ & 2.60e-02 & 1.78 & 3.78e-01 & 1.79 & $8$  & 7.92e-04 & 2.02 &1.31e-02  & 1.97\\ 
$4$ & 1.24e-02 & 2.21 & 1.82e-01 & 1.97 &&&&&\\
 \end{tabular}
 \caption{$L^2$-error, $L^\infty$-error and the corresponding EOCs for Example \ref{Ex:Sphere}.}\label{T:Sphere}
 \end{table}
 
Figure \ref{Fig:Sphere} shows the solution of $(\mathbb P_d^h)$ at different points in time. Note that the white line marks the border between active and inactive sets. On the active parts, the optimal control assumes the value $a$ or $b$, respectively.

Let us conclude with an example for $(\mathbb P_T^h)$ with a desired state $y_T$ that just barely lies in $L^2(\Gamma(T))$. In this situation we can only expect $\mathcal O(h)$-convergence. We consider the unconstrained problem
\begin{Example}[Low Regularity]\label{Ex:LR}
Consider problem $(\mathbb P_T)$ with $\alpha = 1$, $a=-\infty$, $b=\infty$, $T=1$ and $\Gamma(t)$ as in Example \ref{Ex:Sphere}. Let $y_T=\frac{1}{(x+y)^{0.45}}$.
\end{Example}
 Since we do not know the exact solution of Example \ref{Ex:LR}, we estimate the relative error by $ERR_{L^2}^i\simeq \|\bar u_i^l-\bar u_{i+2}\|_{L^2_{L^2(\Gamma^{i+2})}}/\|\bar u_{i+2}\|_{L^2_{L^2(\Gamma^{i+2})}}$, where $\bar u_i$ denotes the solution of $(\mathbb P_T^h)$ on the $i$th refinement $\{\Gamma^{i}(t)\}_{t\in[0,T]}$ of $\{\Gamma(t)\}_{t\in[0,T]}$. The lift $(\cdot)^l$ is taken perpendicular to the smooth surface $\Gamma(t)$.
 Table \ref{T:LR} shows the estimated $L^2$-errors and corresponding EOCs. We computed the $L^2(\Gamma^h(T))$-projection $P_T^h {y_T}_l$ analytically. Otherwise the error introduced by the numerical integration of the non-smooth function $y_T$ would be dominant. It helps that all our triangulations resolve the plane $\{x+y=0\}$.
 \begin{table}
 \center
\begin{tabular}{r|ccccc ccc}
R &$0$  & $1$ & $2$ & $3$ & $4$ &$5$  & $6$ & $7$   \\
\hline
$ERR_{L^2}$ &0.1984  &  0.0982   & 0.0771   & 0.0519   & 0.0369 &  0.0265   & 0.0193   & 0.0138     \\ 
$EOC_{L^2}$ & -&-&1.6460 &   1.5501 &   1.3521 &   1.0755 &   0.9928  &  0.9665\\
$H$ &1.6330    &1.1547 &   0.9194 &   0.7654 &   0.5333  &  0.4099  &  0.2769  &  0.2085 
 \end{tabular}
 \caption{$L^2$-error and the corresponding EOC for Example \ref{Ex:LR}. $H$ is the maximal edge length of $\Gamma_0^h$ (both examples). }\label{T:LR}
 \end{table}
 \section*{Acknowledgement}
The author would like to thank Prof. Dziuk for the fruitful discussion during his stay in Hamburg in November 2010, and for kindly providing the preprints \cite{DziukElliott2010_PP} and \cite{DziukElliott2011_PP}.
 
\bibliographystyle{annotate}
\bibliography{/Users/morten/Documents/Bibs/all}

\begin{thebibliography}{HPUU09}

\bibitem[CL55]{CoddingtonLevinson1955}
Earl~A. Coddington and Norman Levinson.
\newblock {\em {Theory of ordinary differential equations.}}
\newblock {New York, Toronto, London: McGill-Hill Book Company}, 1955.


\bibitem[CL91]{CiarletLions1991}
P.~G. Ciarlet and J.~L. Lions, editors.
\newblock {\em {Handbook of numerical analysis. Volume II: Finite element
  methods (Part 1).}}
\newblock {Amsterdam etc.: North-Holland}, 1991.


\bibitem[DE07]{DziukElliott2007}
G.~Dziuk and C.M. Elliott.
\newblock {Finite elements on evolving surfaces.}
\newblock {\em IMA J. Numer. Anal.}, 27(2):262--292, 2007.


\bibitem[DE10]{DziukElliott2010_PP}
G.~Dziuk and C.~M. Elliott.
\newblock L2-estimates for the evolving surface finite element method.
\newblock submitted, 2010.

\bibitem[DE11]{DziukElliott2011_PP}
G.~Dziuk and C.~M. Elliott.
\newblock Fully discrete evolving surface finite element method.
\newblock submitted, 2011.

\bibitem[DLM11]{DziukLubichMansour2010}
G.~Dziuk, Ch. Lubich, and D.~Mansour.
\newblock Runge-kutta time discretization of parabolic differential equations
  on evolving surfaces.
\newblock accepted, 2011.

\bibitem[Dzi88]{Dziuk1988}
G.~Dziuk.
\newblock {Finite elements for the Beltrami operator on arbitrary surfaces.}
\newblock {Partial differential equations and calculus of variations, Lect.
  Notes Math. 1357, 142-155}, 1988.

\bibitem[Eva98]{Evans1998}
L.~C. Evans.
\newblock {\em {Partial differential equations.}}
\newblock {Graduate Studies in Mathematics. 19. Providence, RI: American
  Mathematical Society (AMS)}, 1998.


\bibitem[HIK03]{HintermuellerItoKunisch2003}
M.~Hinterm{\"u}ller, K.~Ito, and K.~Kunisch.
\newblock {The primal-dual active set strategy as a semismooth Newton method.}
\newblock {\em SIAM J. Optim.}, 13(3):865--888, 2003.


\bibitem[Hin05]{Hinze2005}
M.~Hinze.
\newblock {A variational discretization concept in control constrained
  optimization: The linear-quadratic case.}
\newblock {\em Comput. Optim. Appl.}, 30(1):45--61, 2005.


\bibitem[HPUU09]{HinzePinnauUlbrich2009}
M.~Hinze, R.~Pinnau, M.~Ulbrich, and S.~Ulbrich.
\newblock {\em {Optimization with PDE constraints.}}
\newblock {Mathematical Modelling: Theory and Applications 23. Dordrecht:
  Springer}, 2009.


\bibitem[HV10]{HinzeVierling2010_Surf}
M.~Hinze and M.~Vierling.
\newblock Optimal control of the laplace-beltrami operator on compact surfaces
  -- concept and numerical treatment.
\newblock submitted, 2010.

\bibitem[HV11]{HinzeVierling2010}
M.~Hinze and M.~Vierling.
\newblock A globalized semi-smooth newton method for variational discretization
  of control constrained elliptic optimal control problems.
\newblock In {\em Constrained Optimization and Optimal Control for Partial
  Differential Equations}. Birkh{\"a}user, 2011.

\bibitem[Lio71]{Lions1971}
J.~L. Lions.
\newblock {\em {Optimal control of systems governed by partial differential
  equations.}}
\newblock {Translated by S.K. Mitter. (Die Grundlehren der mathematischen
  Wissenschaften. Band 170.) Berlin-Heidelberg-New York: Springer-Verlag },
  1971.


\bibitem[LM68]{LionsMagenes1968}
J.~L. Lions and E.~Magenes.
\newblock {\em {Probl\`emes aux limites non homogenes et applications. Vol. 1,
  2.}}
\newblock {Paris: Dunod }, 1968.


\bibitem[LSU68]{Ladyzhenskaya1968}
O.A. Ladyzhenskaya, V.A. Solonnikov, and N.N. Ural'tseva.
\newblock {\em {Linear and quasi-linear equations of parabolic type. Translated
  from the Russian by S. Smith.}}
\newblock {Translations of Mathematical Monographs. 23. Providence, RI:
  American Mathematical Society (AMS)}, 1968.


\bibitem[MMV98]{MartioMiklyukovVuorinen1998}
O.~Martio, V.M. Miklyukov, and M.~Vuorinen.
\newblock {Morrey's lemma on Riemannian manifolds.}
\newblock {\em Rev. Roum. Math. Pures Appl.}, 43(1-2):183--210, 1998.


\bibitem[MV08a]{MeidnerVexler2008_I}
D.~Meidner and B.~Vexler.
\newblock {A priori error estimates for space-time finite element
  discretization of parabolic optimal control problems. I: Problems without
  control constraints.}
\newblock {\em SIAM J. Control Optim.}, 47(3):1150--1177, 2008.


\bibitem[MV08b]{MeidnerVexler2008_II}
Dominik Meidner and Boris Vexler.
\newblock {A priori error estimates for space-time finite element
  discretization of parabolic optimal control problems. II: Problems with
  control constraints.}
\newblock {\em SIAM J. Control Optim.}, 47(3):1301--1329, 2008.


\bibitem[Sch10]{Schumacher2010}
A.~Schumacher.
\newblock Die W\"armeleitungsgleichung auf bewegten Oberfl{\"a}chen.
\newblock Master's thesis, Universit\"at Freiburg, 2010.

\bibitem[SD05]{SinhaDeka2005}
R.~K. Sinha and B.~Deka.
\newblock {Optimal error estimates for linear parabolic problems with
  discontinuous coefficients.}
\newblock {\em SIAM J. Numer. Anal.}, 43(2):733--749, 2005.


\bibitem[Tr{\"o}05]{Troeltzsch2005}
F.~Tr{\"o}ltzsch.
\newblock {\em {Optimal control of partial differential equations. Theory,
  procedures, and applications. (Optimale Steuerung partieller
  Differentialgleichungen. Theorie, Verfahren und Anwendungen.)}}.
\newblock {Wiesbaden: Vieweg}, 2005.


\bibitem[Ulb03]{Ulbrich2003}
M.~Ulbrich.
\newblock {Semismooth Newton methods for operator equations in function
  spaces.}
\newblock {\em SIAM J. Optim.}, 13(3):805--841, 2003.


\bibitem[Vie07]{VierlingDiplomaThesis}
M.~Vierling.
\newblock Ein semiglattes Newtonverfahren f{\"u}r semidiskretisierte
  steuerungsbeschr{\"a}nkte Optimalsteuerungsprobleme.
\newblock Master's thesis, Universit{\"a}t Hamburg, 2007.

\end{thebibliography}
\end{document}